\documentclass[11pt,reqno,a4paper,preprint]{elsarticle}
\usepackage{amsmath,amssymb,amsthm}
\usepackage{mathtools}
\usepackage{enumitem}
\usepackage[margin=3cm,top=4cm,bottom=4cm,footskip=1cm,headsep=2cm]{geometry}
\usepackage[utf8]{inputenc}
\usepackage{float}
\usepackage{amsthm}
\usepackage[british]{babel}
\usepackage{color}

\title{A second-order structure-preserving discretization for the Cahn-Hilliard/Allen-Cahn system with cross-kinetic coupling}

\def\div{\operatorname{div}}

\def\RR{\mathbb{R}}

\def\LL{\mathbf{L}}

\def\hbmu{\hat{\bar{\mu}}}
\def\hbrho{\hat{\bar{\rho}}}
\def\hbeta{\hat{\bar{\eta}}}
\def\Th{\mathcal{T}_h}
\def\Itau{\mathcal{I}_\tau}
\def\I{I}

\def\la{\langle}
\def\ra{\rangle}

\def\err{\operatorname{err}}

\def\Th{\mathcal{T}_h}
\def\Vh{\mathcal{V}_h}

\DeclarePairedDelimiter{\norm}{\|}{\|}
\DeclarePairedDelimiter{\snorm}{|}{|}
\newcommand{\na}{\nabla}
\def\E{\mathcal{E}}
\def\D{\mathcal{D}}

\def\na{\nabla}

\def\bmu{\bar{\mu}}
\def\hbmu{\hat{\bar{\mu}}}
\def\hbrho{\hat{\bar{\rho}}}
\def\hbeta{\hat{\bar{\eta}}}
\def\bLL{\bar{\LL}}
\def\hbLL{\hat{\bar{\LL}}}
\newtheorem{lemma}{Lemma}
\newtheorem{problem}[lemma]{Problem}
\newtheorem{theorem}[lemma]{Theorem}

\theoremstyle{definition}

\newtheorem{remark}[lemma]{Remark}

\def\dt{\partial_t}
\def\dtt{\partial_{tt}}

\def\Vh{\mathcal{V}_h}

\usepackage{xcolor}

\def\changestwo#1{#1}

\begin{document}
\author[1]{Aaron Brunk\corref{cor1}}
\cortext[cor1]{Corresponding author}
\ead{abrunk@uni-mainz.de}
\author[2,3]{Herbert Egger}
\ead{herbert.egger@jku.at}
\author[3]{Oliver Habrich}
\ead{oliver.habrich@jku.at}
\affiliation[1]{organization={Institute of Mathematics, Johannes Gutenberg-University},
city={Mainz},
country={Germany}}
\affiliation[2]{organization={Johann Radon Institute for Computational and Applied Mathematics},city={Linz},country={Austria}}
\affiliation[3]{organization={Institute for Numerical Mathematics, Johannes Kepler University},city={Linz},country={Austria}}

\begin{abstract}
We study the numerical solution of a Cahn-Hilliard/Allen-Cahn system with strong coupling through state and gradient dependent non-diagonal mobility matrices. 
A fully discrete approximation scheme in space and time is proposed which preserves the underlying gradient flow structure and leads to dissipation of the free-energy on the discrete level. 
Existence and uniqueness of the discrete solution is established and relative energy estimates are used to prove optimal convergence rates in space and time under minimal smoothness assumptions. Numerical tests are presented for illustration of the theoretical results and to demonstrate the viability of the proposed methods.
\end{abstract}

\begingroup
\def\uppercasenonmath#1{} 
\let\MakeUppercase\relax 
\maketitle
\endgroup

\begin{quote} 
\end{quote}
\begin{quote}
\noindent
{\small {\bf AMS-classification (2000):}
35K52, 
35K55, 
65M12, 
65M60, 
82C26 
}
\end{quote}

\begin{center} 'Declarations of interest: none' \end{center}

\bigskip

\section*{Introduction}
\label{sec:intro}
Phase-field models involving, at the same time,  conserved and non-conserved quantities have been proposed to describe the simultaneous phase separation and ordering in binary alloys~\cite{cahn1994}. 
%
%
Similar models have been applied recently for modelling phase transformations in solid-state sintering~\cite{Boussinot2013} and, more generally, in the context of grain boundary segregation~\cite{ABDELJAWAD2017528}.  
In this work we study the numerical approximation of coupled systems of Cahn-Hilliard and Allen-Cahn equations of the form 
\begin{alignat}{2}
\dt\rho &= \div(\LL_{11}\nabla\mu_\rho + \LL_{12}\mu_\eta), \qquad & \mu_\rho &= -\gamma_\rho\Delta\rho + f_\rho(\rho,\eta), \label{eq:s1}\\
\dt\eta &= - \LL_{12}\cdot\nabla\mu_\rho - \LL_{22}\mu_\eta, \qquad & \mu_\eta &= -\gamma_\eta\Delta\eta + f_\eta(\rho,\eta).\label{eq:s2}
\end{alignat}
Here $\rho$ and $\eta$ are the conserved and non-conserved quantities, $\gamma_\rho$, $\gamma_\eta$ in turn are the corresponding interface parameters, and $f(\rho,\eta)$ is a free energy density, and $f_\rho=\partial_\rho f$, $f_\eta=\partial_\eta f$ denote the partial derivatives of the function $f(\cdot,\cdot)$.
Furthermore, $\LL$ is a generalized mobility matrix, which is assumed symmetric and positive definite, but may in general depend on the phase fields $\rho$, $\eta$ as well as their gradients. 
The chemical potentials 
$\mu_\rho=\delta_\rho E(\rho,\eta)$, $\mu_\eta=\delta_\eta E(\rho,\eta)$ 
are the variational derivatives of the total free energy 
\begin{equation}
\mathcal{E}(\rho,\eta) := \int_\Omega E(\rho,\eta) := \int_\Omega \tfrac{\gamma_\rho}{2}\snorm{\nabla\rho}^2  + \tfrac{\gamma_\eta}{2}\snorm{\nabla\eta}^2 + f(\rho,\eta), 
\end{equation}
with respect to $\rho$ and $\eta$. As a consequence, \eqref{eq:s1}--\eqref{eq:s2} can be considered as a generalized gradient flow describing the continuous decay of the free energy $E(\rho,\eta)$ along solutions. 

\subsection*{Diagonal mobilities.}%
Most of the results on Cahn-Hilliard/Allen-Cahn systems available in the literature are concerned with the diagonal diffusion case, i.e., $\LL_{12}=\mathbf{0}$. 
For existence results in various settings, see e.g. \cite{BROCHET199483,Miranville2019,DalPasso1999}. 
Sharp interface limits have been studied in \cite{Nurnberg,cahn1996limiting,NOVICKCOHEN20001} and a corresponding numerical method can be found in \cite{Barrett}.
In \cite{Yang} a fully implicit time integration scheme using finite differences is validated by numerical tests. 
The authors of \cite{Xia} introduce an energy-stable local discontinuous Galerkin scheme and show second-order convergence experimentally.
Huang et al. \cite{Huang} propose an energy-stable second order space-time finite-difference scheme using discrete variational calculus in the spirit of averaged vector field methods \cite{Celledoni12,Hairer14}.

\subsection*{Cross-kinetic coupling.}
In order to circumvent spurious effects at the interface \cite{BolladaEtAl2018}, the incorporation of cross-kinetic coupling $\LL_{12} \ne 0$ has been proposed in \cite{Brener2012}. 
The well-posedness of related models has been investigated in \cite{Brunk22}. 
In this paper, we study the systematic numerical approximation of such Cahn-Hilliard/Allen-Cahn systems with cross-kinetic coupling terms. 
Our main contributions can be summarized as follows: 
\begin{itemize}\itemsep1ex
\item an unconditionally energy stable and mass conservative discretization scheme is proposed for the CH/AC system with gradient-dependent cross-coupling;
\item a discrete stability analysis is developed based on relative energy arguments; 
\item order optimal convergence rates are established in the presence of various nonlinearities and under minimal smoothness assumptions on the solution.  
\end{itemize}
To the best of our knowledge, this paper contains the first complete error analysis for a second order approximation of phase-field models with cross-kinetic coupling and gradient dependent mobilities. 
As a complement to the analytical results of the paper, the practicability of the proposed method and the validity of the theoretical results will be illustrated by numerical tests.

\subsection*{Related problems.}
Our method and its convergence analysis can in principle be generalized to related multi-component Cahn-Hilliard or Allen-Cahn systems, for which a wider literature is available. 
The numerical solution of degenerate Cahn-Hilliard systems by a variational inequality approach has been considered in \cite{BarrettBloweyGarcke01,ZhouXie23}. 
Various splitting methods have been proposed in \cite{Lee12,Li22,YangKim21} to accelerate the numerical solution of the coupled nonlinear systems. 
A rigorous numerical analysis for a first order scheme has been presented in \cite{BarrettBlowey99}.
In \cite{KimKang09}, a second order numerical scheme has been investigated for the solution of a ternary Cahn-Hilliard system, which is somewhat related to the method proposed in this paper, and convergence rates have been demonstrated numerically. 

\subsection*{Outline.}
The remainder of the manuscript is organized as follows: 
Section~\ref{sec:prelim} presents our notation and basic assumptions,
and the basic ingredients for our discretization strategy.
In Section~\ref{sec:main}, we then introduce our numerical method and state our main theoretical results. 
Section~\ref{sec:stab} is concerned with the stability of discrete solutions, which is the  key ingredient for our error analysis presented in Section~\ref{sec:err}.
Some auxiliary tools are summarized in the appendix.
For illustration of our theoretical results, we present some numerical tests in Section~\ref{sec:num}, and the paper closes with a short discussion.
%

\section{Preliminaries} \label{sec:prelim}

Before we present our discretization method and main results in detail, let us briefly introduce our notation and main assumptions, and recall some basic facts.

\subsection{Notation}
The system \eqref{eq:s1}--\eqref{eq:s2} is investigated on a finite time interval $(0,T)$. 
To avoid the discussion of boundary conditions, we consider a spatially periodic setting, i.e., 
\begin{itemize}
\item[(A0)] $\Omega \subset \RR^d$, $d=2,3$ is a cube and identified with the $d$-dimensional torus $\mathcal{T}^d$.\\
Moreover, functions on $\Omega$ are assumed to be periodic throughout the paper. 
\end{itemize}
By $L^p(\Omega)$, $W^{k,p}(\Omega)$, we denote the corresponding Lebesgue and Sobolev spaces of periodic functions with norms $\norm{\cdot}_{L^{p}}$ and $\norm{\cdot}_{W^{k,p}}$. 
As usual, we abbreviate $H^k(\Omega)=W^{k,2}(\Omega)$ and write $\norm{\cdot}_{H^{k}} = \norm{\cdot}_{W^{k,2}}$. 
For functions $r \in L^2(\Omega)$, we define the dual norm 
\begin{align} \label{eq:dualnorm}
    \norm{r}_{H^{-k}} = \sup_{v \in H^k(\Omega)} \frac{\la r, v\ra}{\|v\|_{k}}.
\end{align}
Here $\langle \cdot, \cdot\rangle$ denotes the scalar product on $L^2(\Omega)$, which is defined by
\begin{align*}
\la u, v \ra = \int_\Omega u \cdot v  \qquad \forall u,v \in L^2(\Omega).    
\end{align*}
By $L^2_0(\Omega) \subset L^2(\Omega)$, we denote the spaces of square integrable functions with zero average.
As usual, we denote by $L^p(a,b;X)$, $W^{k,p}(a,b;X)$, and
$H^k(a,b;X)$, the Bochner spaces of integrable or differentiable functions on the time interval $(a,b)$ with values in some Banach space $X$. If $(a,b)=(0,T)$, we omit reference to the time interval and briefly write $L^p(X)$. The corresponding norms are denoted, e.g., by $\|\cdot\|_{L^p(X)}$ or $\|\cdot\|_{H^k(X)}$.
%

\subsection{Assumptions on the parameters}

Throughout the paper, we assume that the model parameters are sufficiently smooth and satisfy some typical conditions, i.e.,
\begin{itemize}
    \item[(A1)] the interface parameters $\gamma_\rho,\gamma_\eta$ are positive constants;
    \item[(A2)] for any choice of $\omega=(\rho,\eta,\nabla\rho,\nabla\eta)$, the matrix $\LL(\omega) \in\mathbb{R}^{(d+1) \times (d+1)}$ is symmetric and positive definite with 
    \begin{equation*}
      \lambda_1\snorm{\xi}^2 \leq \mathbf{\xi}^\top\LL(\omega)\mathbf{\xi} \leq \lambda_2\snorm{\xi}^2, \quad \forall \mathbf{\xi}\in\mathbb{R}^{d+1}.
    \end{equation*}
    Furthermore, every component of $\LL(\cdot)$ is a $C^2$ function of its arguments $\omega$, with derivatives uniformly bounded by some constant 
    $\lambda_3$;
    \item[(A3)] the potential $f(\cdot,\cdot)$ is smooth with $f(\rho,\eta)\geq 0$ and satisfies
    \begin{align*}
        \snorm*{\frac{\partial^{k+\ell} f(\rho,\eta)}{\partial^k\rho \, \partial^\ell\eta}} &\leq C_1 \sum |\rho|^{4-k} + |\eta|^{4-\ell} + C_2
    \end{align*}
    for all $0 \le k,\ell,k+\ell \le 4$. 
    Furthermore, the shifted potential
    \begin{equation*}
        f(\rho,\eta) + \frac{\alpha}{2}(\snorm{\rho}^2 + \snorm{\eta}^2) 
    \end{equation*}
    is strictly convex for some $\alpha>0$. 
\end{itemize}
The assumptions in (A3) essentially encode that $f$ is sufficiently smooth and, together with its derivatives, satisfies appropriate growth conditions. 

\subsection{Variational characterization}

Any sufficiently smooth periodic solution of \eqref{eq:s1}--\eqref{eq:s2} on $\Omega \times (0,T)$ can be seen to satisfy the variational identities
\begin{align}
 \la \dt\rho,v_1 \ra + \la \LL_{11}(\omega)\nabla\mu_\rho,\nabla v_1 \ra + \la \mu_\eta\LL_{12}(\omega),\nabla v_1 \ra &= 0, \label{eq:1}\\
 \la \dt\eta,v_2 \ra + \la v_2\LL_{12}(\omega),\nabla\mu_\rho \ra + \la \LL_{22}(\omega)\mu_\eta,v_2 \ra &= 0,\label{eq:2}\\
 \la \mu_\rho,w_1 \ra - \gamma_\rho\la \nabla\rho,\nabla w_1 \ra - \la f_\rho(\rho,\eta),w_1 \ra &= 0, \label{eq:3}\\
 \la \mu_\eta,w_2 \ra - \gamma_\eta\la \nabla\eta,\nabla w_2 \ra - \la f_\eta(\rho,\eta),w_2 \ra &= 0, \label{eq:4}
\end{align}
for $0 \le t \le T$ and sufficiently regular periodic test functions $v_1$, $v_2$ and $w_1$, $w_2$. 
Note that the solution components depend on time $t$, while the test functions are independent of $t$. Hence the variational identities \eqref{eq:1}--\eqref{eq:4} have to be understood pointwise in time. 
Moreover, the symbol $\omega=(\rho,\eta,\nabla\rho,\nabla \eta)$ is again used for abbreviation. \changestwo{If $\dt\rho$ is not regular enough, i.e. $\dt\rho\notin L^1(0,T;L^1(\Omega))$, the inner product $\la \dt\rho,v_1 \ra$ has to be replaced by a dual pairing, c.f. \cite{Feng} for such a paring.}
%

\subsection{Basic properties}

The variational identities \eqref{eq:1}--\eqref{eq:4} allow us to immediately establish some important properties of solutions:
By testing \eqref{eq:1} with $v_1 \equiv 1$, we get 
\begin{align} \label{eq:mass}
\frac{d}{dt} \int_\Omega \rho  = 0,
\end{align}
which encodes the conservation of mass. 
By formal differentiation of the energy $\mathcal{E}(\rho,\eta)$ along a solution in time, we obtain
\begin{align} \label{eq:dissi}
\frac{d}{dt} \mathcal{E}(\rho,\eta) &= \la \dt\rho ,\delta_\rho \mathcal{E}(\rho,\eta) \ra + \la \dt\eta ,\delta_\eta \mathcal{E}(\rho,\eta) \ra  
= \la \dt\rho ,\mu_\rho \ra + \la \dt\eta ,\mu_\eta \ra \\
&= - \la \LL_{11}(\omega)\nabla\mu_\rho,\nabla \mu_\rho \ra - \la \mu_\eta\LL_{12}(\omega),\nabla \mu_\rho \ra -  \la \mu_\eta\LL_{12}(\omega),\nabla\mu_\rho \ra - \la \LL_{22}(\omega)\mu_\eta,\mu_\eta \ra. \notag 
\end{align}
In the second identity, we here used the variational equations \eqref{eq:3} and \eqref{eq:4} with $w_1=\dt\rho$ and $w_2=\dt\eta$, and in the third step, we used \eqref{eq:1} and \eqref{eq:2} with test functions $v_1=\mu_\rho$ and $v_2=\mu_\eta$.
The identity \eqref{eq:dissi} can further be written compactly as 
\begin{align} \label{eq:dissi2}
\frac{d}{dt} \E(\rho,\eta) = -\mathcal{D}_{\rho,\eta}(\mu_\rho,\mu_\eta)
\end{align}
with dissipation functional $\mathcal{D}_{\rho,\eta}(\mu_\rho,\mu_\eta)=\la \begin{pmatrix} \nabla \bar\mu_{\rho} \\ \bar\mu_{\eta}\end{pmatrix},\bLL \begin{pmatrix} \nabla \bar\mu_{\rho} \\ \bar\mu_{\eta}\end{pmatrix}\ra$, which is non-negative since we assumed positive definiteness of the mobility matrix $\LL(\omega)$.
This identity encodes the underlying energy-dissipation principle of the problem, and hence implies the thermodynamic consistency of the model under investigation.

\section{Discretization scheme and main results}  \label{sec:main}

In the following, we introduce a fully practical numerical approximation scheme for our model problem \eqref{eq:s1}--\eqref{eq:s2}, which preserves the basic conservation and dissipation properties of the problem on the discrete level and which yields order optimal error estimates. 

\subsection{Space discretization}
Let us start with introducing the most relevant notation concerning the finite element discretization in space. We consider 
\begin{itemize}\itemsep1ex
\item[(A4)] a geometrically-conforming quasi-uniform partition $\Th$ of $\Omega$ into simplices that can be extended periodically to periodic extensions of $\Omega$. 
\end{itemize}
By quasi-uniform, we mean that there exists a constant $\sigma>0$ such that $\sigma h \le \rho_K \le h_K \le h$ for all $K \in \Th$, where $\rho_K$ and $h_K$ are the inner-circle radius and diameter of the element $K \in \Th$ and $h=\max_{K \in \Th} h_T$ is the global mesh size \cite{BrennerScott}. 
We then denote by 
\begin{align*}
\Vh := \{v \in H^1(\Omega) : v|_K \in P_2(K) \quad \forall K \in \Th\}
\end{align*}
the space of continuous piecewise quadratic functions over $\Th$. 
We write $\pi_h^0 : L^2(\Omega) \to \Vh$ and $\pi_h^1 : H^1(\Omega) \to \Vh$ for the $L^2$- and $H^1$-orthogonal projections, which are defined by
\begin{alignat}{2}
\la \pi_h^0 u - u, v_h\ra &= 0 \qquad && \forall v_h \in \Vh, \label{eq:defl2proj}\\
\la \pi_h^1 u - u, v_h\ra + \la \nabla (\pi_h^1 u - u ), \nabla v_h \ra &= 0 \qquad && \forall v_h \in \Vh.\label{eq:defh1proj}
\end{alignat}
For the convenience of the reader and later reference, some well-known approximation properties of these projection operators are summarized in Appendix~\ref{asec:space}.

\subsection{Time discretization}
We also employ piecewise polynomial functions for the approximation in time. 
For ease of presentation, we consider a uniform grid 
\begin{itemize}\itemsep1ex
\item[(A5)] $\Itau:=\{0=t^0,t^1,\ldots,t^N=T\}$ \quad with time steps $t^n = n \tau$ and $\tau = T/N$.
\end{itemize}
More general non-uniform grids could be treated with minor modifications.
By 
\begin{align}
P_k(\Itau;X) 
\qquad \text{and} \qquad 
P_k^c(\Itau;X) = P_k(\Itau;X) \cap C(0,T;X),
\end{align}
we denote the spaces of discontinuous respectively continuous piecewise polynomial functions of degree $\le k$ over the time grid $\Itau$, with values in some vector space $X$.
We utilize a bar symbol $\bar u$ to denote functions in $P_0(\Itau;X)$ that are piecewise constant in time.

In our analysis, we use the piecewise linear interpolation
$\I_\tau^1:H^1(0,T;X)\to P_1^c(\Itau;X)$ as well as the $L^2$-orthogonal projection
$\bar \pi_\tau^0 : L^2(0,T;X) \to \Pi_0(\Itau;X)$ to piecewise constants in time. 
Some important properties of these operators are again summarized in Appendix~\ref{asec:time}. 
With a slight abuse of notation, we will frequently use the symbol 
\begin{align} \label{eq:proj}
\bar u = \bar \pi_\tau^0 u 
\end{align}
also to abbreviate the piecewise constant projection of a function $u \in L^2(X)$ in time. 
As a final ingredient, we introduce the abbreviations $I_n=(t^{n-1},t^n)$ and
\begin{align} \label{eq:abn}
    \la a, b\ra^n := \int_{t^{n-1}}^{t^n} \la a(s), b(s) \ra \,ds = \int_{I_n} \la a(s), b(s) \ra \,ds
\end{align}
for the individual time intervals and the corresponding integrals.

\subsection{Numerical method}
We can now formulate our discretization scheme, which is based on an inexact Galerkin approximation of the variational identities \eqref{eq:1}--\eqref{eq:4}. 
\begin{problem}\label{prob:pg}
Let $\rho_{h,0},\eta_{h,0} \in \Vh$ be given initial values. 
Then find $\rho_{h,\tau},\eta_{h,\tau} \in P^c_1(\Itau;\Vh)$ and $\bar\mu_{\rho,h,\tau},\bar\mu_{\eta,h,\tau}\in P_0(\Itau;\Vh)$ such that $\rho_{h,\tau}(0)=\rho_{h,0}$, $\eta_{h,\tau}=\eta_{h,0}$, and such that   
\begin{align}
 \la \dt\rho_{h,\tau},\bar v_{1,h,\tau} \ra^n + \la \LL_{11}(\bar \omega_{h,\tau}) \nabla\bar\mu_{\rho,h,\tau} + \bar\mu_{\eta,h,\tau} \LL_{12}(\bar \omega_{h,\tau}),\nabla\bar v_{1,h,\tau} \ra^n  &= 0, \label{eq:pg1}\\
 \la \dt\eta_{h,\tau},\bar  v_{2,h,\tau} \ra^n + \la \LL_{12}(\bar \omega_{h,\tau})\nabla\bar\mu_{\rho,h,\tau}+\LL_{22}(\bar \omega_{h,\tau})\bar\mu_{\eta,h,\tau},\bar v_{2,h,\tau} \ra^n &= 0,\label{eq:pg2}\\
 \la \bar\mu_{\rho,h,\tau},\bar w_{1,h,\tau} \ra^n - \gamma_\rho\la \nabla\bar\rho_{h,\tau},\nabla \bar w_{1,h,\tau} \ra^n - \la f_\rho(\rho_{h,\tau},\eta_{h,\tau}),\bar w_{1,h,\tau} \ra^n &= 0, \label{eq:pg3}\\
 \la \bar\mu_{\eta,h,\tau},\bar w_{2,h,\tau} \ra^n - \gamma_\eta\la \nabla\bar\eta_{h,\tau},\nabla \bar w_{2,h,\tau} \ra^n - \la f_\eta(\rho_{h,\tau},\eta_{h,\tau}),\bar w_{2,h,\tau} \ra^n &= 0, \label{eq:pg4}    
\end{align}
hold for all $\bar v_{1,h,\tau}, \bar v_{2,h,\tau},\bar w_{1,h,\tau},\bar w_{2,h,\tau}\in P_0(I_n;\Vh)$ and all time steps $1 \le n\le N$. 
\end{problem}
Similar to before, the symbol $\omega_{h,\tau}=(\rho_{h,\tau},\eta_{h,\tau},\nabla \rho_{h,\tau},\nabla \eta_{h,\tau})$ here is used for abbreviation and $\bar \omega_{h,\tau} = \bar\pi^0_\tau \omega_{h,\tau}$ denotes the piecewise constant projection in time.  \changestwo{Similar as before  $f_\rho(\rho_{h,\tau},\eta_{h,\tau}), f_\eta(\rho_{h,\tau},\eta_{h,\tau})$ denote the partial derivatives of $f$ after the first and second argument evaluated at $\rho_{h,\tau},\eta_{h,\tau}.$}

\subsection{Main results}

We next formulate our main theorems. 
The first result is concerned with the well-posedness of the discretization scheme and the properties of its solutions.

\begin{theorem}[Existence and properties of discrete solutions] \label{thm:well-posed}$ $\\
Let (A0)--(A5) hold. Then for any $h,\tau>0$ and any choice $\rho_{h,0},\eta_{h,0} \in \Vh$ of initial values, Problem~\ref{prob:pg} has at least one solution.  
Moreover, any such discrete solution satisfies
\begin{align*}
\frac{d}{dt} \int_\Omega \rho_{h,\tau}  = 0
\qquad \text{and} \qquad 
\E(\rho_{h,\tau},\eta_{h,\tau})\Big\vert_{t^{n-1}}^{t^n} = - \int_{I^n} \D_{\bar \rho_{h,\tau},\bar \eta_{h,\tau}}(\bar\mu_{\rho,h,\tau},\bar\mu_{\eta,h,\tau}) \, ds,
\end{align*}
i.e. mass is conserved and energy is dissipated for all time steps $1 \le n \le N$. 
\end{theorem}
\begin{proof}
The two identities are obtained almost verbatim as those on the continuous level by appropriate testing of the underlying variational identities.
As a consequence of the second identity and Assumptions~(A1)--(A3), any solution of Problem~\ref{prob:pg}\eqref{eq:pg1}--\eqref{eq:pg4} satisfies uniform a-priori bounds. 
The existence of a solution can then be obtained from the 
Leray-Schauder principle, see \cite[Theorem~6A]{Zeidler1} as follows: In the $n$th time step, we seek to determine $x=(\dt \rho_{h,\tau}(t^{n-1/2}),\dt \eta(t^{n-1/2}),\bar \mu_{\rho,h,\tau}(t^{n-1/2}),\bar\mu_{\eta,h,\tau}(t^{n-1/2}))$; afterwards, we update $\rho_{h,\tau}(t^n)=\rho_{h,\tau}(t^{n-1})+\tau \dt \rho_{h,\tau}(t^{n-1/2})$ and $\eta_{h,\tau}(t^n)=\eta_{h,\tau}(t^{n-1}) + \tau \dt \eta_{h,\tau}(t^{n-1/2})$ and proceed. 
By elementary manipulations, we may cast \eqref{eq:pg1}--\eqref{eq:pg4} into a fixed-point equation $x = T(x)$ in a finite dimensional space $X \simeq \RR^N$. The operator $T$ is continuous and compact. 
With the previous considerations, one can see that any solution of $x= \lambda T(x)$ for $0 \le \lambda \le 1$ is bounded uniformly by $\|x\| \le R$ with constant $R>0$ independent of $\lambda$. 
This then implies existence of a solution to $x=T(x)$ and hence to \eqref{eq:pg1}--\eqref{eq:pg4} in the $n$th time step of Problem~\ref{prob:pg}, and we concluded by induction over $n$.
\end{proof}

\subsection*{Convergence rates}
In order to derive quantitative error estimates, we need to assume sufficient regularity of the solution of \eqref{eq:s1}--\eqref{eq:s2}. 
We thus require that
\begin{itemize}
\item[(A6)] a sufficiently regular solution $(\rho,\eta,\mu_\rho,\mu_\eta)$ of \eqref{eq:1}--\eqref{eq:4} exists, satisfying
\begin{alignat*}{2}
\qquad \quad
\rho&\in H^{2}(0,T;H^1(\Omega))\cap H^1(0,T;H^3(\Omega)),\\
\eta&\in H^{2}(0,T;H^1(\Omega))\cap H^1(0,T;H^3(\Omega)), \\
\mu_\rho& \in H^2(0,T;H^1(\Omega))\cap  L^2(0,T;H^{3}(\Omega))\cap L^\infty(0,T;W^{1,\infty}(\Omega)),
\\
\mu_\eta& \in H^2(0,T;L^2(\Omega))\cap  L^2(0,T;H^{2}(\Omega))\cap L^\infty(0,T;L^\infty(\Omega)).
\end{alignat*}
\end{itemize}
All norms of solution components and their derivatives arising in these spaces are thus assumed bounded uniformly. These assumptions are required to ensure sufficient approximation properties of the underlying discretization spaces in space and time, and they allow us to prove our second and main result.
\begin{theorem}[Convergence rates and uniqueness]
\label{thm:fulldisk} $ $\\
Let (A0)--(A6) hold, $0<h,\tau \le c$ be sufficiently small, and let the initial values be chosen according to $\rho_{h,0} = \pi_h^1 \rho(0)$, $\eta_{h,0}=\pi_h^1 \eta(0)$. 
Then the solution of Problem~\ref{prob:pg} satisfies
\begin{align}\label{eq:con_res_1}
\|\rho - \rho_{h,\tau}\|_{L^\infty(H^1)} &+ \|\eta - \eta_{h,\tau}\|_{L^\infty(H^1)} \\
&+ \|\bar\mu_\rho - \bmu_{\rho,h,\tau}\|_{L^2(H^1)} + \|\bar\mu_\eta - \bmu_{\eta,h,\tau}\|_{L^2(L^2)}
\leq C \, (h^2 + \tau^2),\notag
\end{align}
with a constant $C$ depending only the bounds in the assumptions, but not on $h$ and $\tau$.
Let us recall that $\bar \mu_\rho = \bar \pi_\tau^0 \mu_\rho, \bar \mu_\eta = \bar \pi_\tau^0 \mu_\eta$  denote the piecewise constant averages in time.
For the choice $\tau=c' h$ and $h,\tau \le c$ sufficiently small, the discrete solution is unique.
\end{theorem}
The key ingredient for the proof of these results is a careful stability estimate for the discrete problem, see Theorem~\ref{thm:stab}, which is derived via relative energy estimates. 
%

\section{Discrete stability}\label{sec:stab}

In our convergence analysis, we will compare discrete solutions of Problem~\ref{prob:pg} with auxiliary functions $(\hat\rho_{h,\tau},\hat\eta_{h,\tau},\hbmu_{\rho,h,\tau},\hbmu_{\eta,h,\tau})$ that are obtained as certain projections of the true solution. When inserting such functions into the discrete variational problem, we obtain
\begin{align}
\la \dt\hat\rho_{h,\tau},\bar v_{1,h,\tau} \ra^n + \la \bLL_{11,h,\tau}\nabla\hbmu_{\rho,h,\tau}+\hbmu_{\eta,h,\tau}\bLL_{12,h,\tau},\nabla\bar v_{1,h,\tau} \ra^n  &= \la \bar r_{1,h,\tau},\bar v_{1,h,\tau} \ra^n \label{eq:ppg1}\\
\la \dt\hat\eta_{h,\tau},\bar v_{2,h,\tau} \ra^n +  \la\bLL_{12,h,\tau}\nabla\hbmu_{\rho,h,\tau} + \bLL_{22,h,\tau}\hbmu_{\eta,h,\tau},\bar v_{2,h,\tau} \ra^n &= \la \bar r_{2,h,\tau},\bar v_{2,h,\tau} \ra^n \label{eq:ppg2}\\
 \la \hbmu_{\rho,h,\tau},\bar w_{1,h,\tau} \ra^n - \gamma_\rho\la \nabla\hbrho_{h,\tau},\nabla \bar w_{1,h,\tau} \ra^n - \la f_\rho(\hat\rho_{h,\tau},\hat\eta_{h,\tau}),\bar w_{1,h,\tau} \ra^n &= \la \bar r_{3,h,\tau},\bar w_{1,h,\tau} \ra^n \label{eq:ppg3}\\
\la \hbmu_{\eta,h,\tau},\bar w_{2,h,\tau} \ra^n - \gamma_\eta\la \nabla\hbeta_{h,\tau},\nabla \bar w_{2,h,\tau} \ra^n - \la f_\eta(\hat\rho_{h,\tau},\hat\eta_{h,\tau}),\bar w_{2,h,\tau} \ra^n &= \la \bar r_{4,h,\tau},\bar w_{2,h,\tau} \ra^n \label{eq:ppg4}    
\end{align}
for all $\bar v_{1,h,\tau},\bar v_{2,h,\tau},\bar w_{1,h,\tau},\bar w_{2,h,\tau}\in P_0(I_n;\Vh)$ and $n \le N$ with residuals $\bar r_{i,h,\tau} \in P_0(I_\tau;\Vh)$.
\begin{remark}
Let us emphasize that we use the mobility matrix $\bar \LL_{h,\tau}=\LL(\bar \omega_{h,\tau})$ evaluated at the exact discrete solution $\omega_{h,\tau}=(\rho_{h,\tau},\eta_{h,\tau},\nabla \rho_{h,\tau},\nabla \eta_{h,\tau})$ of Problem~\ref{prob:pg} here, which can be understood as a linearization around the discrete solution.
\end{remark}

\subsection{Relative energy}
To measure the difference of two pairs of functions $(\rho,\eta)$ and $(\hat \rho,\hat \eta)$, we use the regularized relative energy 
\begin{align*}
    \mathcal{E}_\alpha(\rho,\eta|\hat\rho,\hat\eta) := \int_\Omega &\frac{\gamma_\rho}{2}\snorm{\nabla(\rho-\hat\rho)}^2 + \frac{\gamma_\eta}{2}\snorm{\nabla(\eta-\hat\eta)}^2 + \frac{\alpha}{2}\snorm{\rho-\hat\rho}^2 + \frac{\alpha}{2}\snorm{\eta-\hat\eta}^2\\
    &+ f(\rho,\eta) - f(\hat\rho,\hat\eta) - f_\rho(\hat\rho,\hat\eta)(\rho-\hat\rho) - f_\eta(\hat\rho,\hat\eta)(\eta-\hat\eta).
\end{align*}
Differences in the chemical potentials $(\mu_\rho,\mu_\eta)$ and $(\hat \mu_\rho,\hat \mu_\eta)$ can be estimated by the dissipation functional $\D_{\rho,\eta}(\mu_\rho-\hat \mu_\rho,\mu_\eta-\hat\mu_\eta)$. 
Due to our assumptions (A1)--(A3), these functionals allow us to estimate the distance between solutions. \changestwo{We recall the notation of $f_\rho(\hat\rho,\hat\eta), f_\eta(\hat\rho,\hat\eta)$ as partial derivatives of $f$ after the first and second argument evaluated at $\hat\rho,\hat\eta.$}

\begin{lemma}\label{lem:properties}
Let (A1)--(A3) hold. Then 
there exists $c_\alpha$, $C_\alpha$, $C_L >0$, such that 
\begin{align*}
  c_\alpha (\norm{\rho-\hat\rho}_{H^1}^2 + \norm{\eta-\hat\eta}_{H^1}^2)  &\leq \E_\alpha(\rho,\eta|\hat\rho,\hat\eta ) \le C_\alpha (\norm{\rho-\hat\rho}_{H^1}^2 + \norm{\eta-\hat\eta}_{H^1}^2), \\
  \norm{\nabla(\mu_\rho-\hat \mu_\rho)}_{L^2}^2 + \norm{\mu_\eta-\hat \mu_\eta}_{L^2}^2 &\leq C_L\D_{\rho,\eta}(\mu_\rho-\hat\mu_\rho,\mu_\eta-\hat\mu_\eta)
\end{align*}
for any choice of $\rho,\eta,\mu_\rho,\hat \rho,\hat \eta,\hat \mu_\rho \in H^1(\Omega)$ and any function $\mu_\eta,\hat\mu_\eta \in L^2(\Omega)$. 
\end{lemma}
\begin{proof}
The first inequality follows from the construction of the relative energy and assumptions (A1) and (A3). The second inequality follows directly from (A2).
\end{proof}

\subsection{Relative energy estimate}
The following stability estimate allows us to estimate differences of discrete solutions in terms of the residuals in \eqref{eq:ppg1}--\eqref{eq:ppg4}.  
\begin{theorem}
\label{thm:stab}
Let (A0)--(A5) hold, let $(\rho_{h,\tau},\bar\mu_{\rho,h,\tau},\eta_{h,\tau},\bar\mu_{\eta,h,\tau})$ be a solution of Problem~\ref{prob:pg}, and $(\hat\rho_{h,\tau},\hbmu_{\rho,h,\tau},\hat\eta_{h,\tau},\hbmu_{\eta,h,\tau})$ solve \eqref{eq:ppg1}--\eqref{eq:ppg4} with residuals $\bar r_{i,h,\tau} \in P_0(I_\tau;\Vh)$. Then 
\begin{align*}
\E_\alpha(\rho_{h,\tau},\eta_{h,\tau}&|\hat\rho_{h,\tau},\hat\eta_{h,\tau})\Big\vert_{t^{n-1}}^{t^n} + C_1 \int_{I_n} \D_{\bar\rho_{h,\tau},\bar\eta_{h,\tau}}(\mu_{\rho,h,\tau}-\hat\mu_{\rho,h,\tau},\mu_{\eta,h,\tau}-\hat\mu_{\eta,h,\tau}) \, ds \\
&\leq C_2\int_{I_n}\E_\alpha(\rho_{h,\tau},\eta_{h,\tau}|\hat\rho_{h,\tau},\hat\eta_{h,\tau})\, ds \\
&\qquad \qquad \qquad + C_3\int_{I_n}\norm{\bar r_{1,h,\tau}}_{H^{-1}}^2 + \norm{\bar r_{2,h,\tau}}_{H^1}^2  + \norm{\bar r_{3,h,\tau}}_{L^2}^2 + \norm{\bar r_{4,h,\tau}}_{L^2}^2 \, ds
\end{align*}
with positive constants $C_1$, $C_2$, $C_3$ depending only on the bounds in the assumptions. 
\end{theorem}
\begin{proof}
For ease of presentation, we omit the subscripts $h,\tau$ in the notation for the discrete functions. 
By differentiation of the relative energy in time, we then get
\begin{align*}
\E(\rho,\eta|\hat\rho,\hat\eta)\vert_{t^{n-1}}^{t^n} &= \int_{I_n} \tfrac{d}{ds}\E(\rho,\eta|\hat\rho,\hat\eta) \, ds\\
&= \gamma_\rho\la \na(\rho-\hat\rho),\na \dt(\rho-\hat\rho) \ra^n  + \la f_\rho(\rho,\eta)- f_\rho(\hat\rho,\hat\eta),\dt(\rho-\hat\rho) \ra^n \\
&\qquad +\gamma_\eta\la \na(\eta-\hat\eta),\na \dt(\eta-\hat\eta) \ra^n  + \la f_\eta(\rho,\eta)- f_\eta(\hat\rho,\hat\eta),\dt(\eta-\hat\eta) \ra^n \\
&\qquad + \la f_\rho(\rho,\eta) - f_\rho(\hat\rho,\hat\eta) - f_{\rho\rho}(\hat\rho,\hat\eta)(\rho-\hat\rho) - f_{\rho\eta}(\hat\rho,\hat\eta)(\eta-\hat\eta), \dt\hat\rho\ra^n \\
&\qquad + \la f_\eta(\rho,\eta) - f_\eta(\hat\rho,\hat\eta) - f_{\eta\rho}(\rho,\eta)(\rho-\hat\rho) - f_{\eta\eta}(\hat\rho,\hat\eta)(\eta-\hat\eta), \dt\hat\eta\ra^n \\
&\qquad +\alpha\la \rho-\hat\rho, \dt(\rho-\hat\rho)\ra^n + \alpha\la \eta-\hat\eta, \dt(\eta-\hat\eta)\ra^n \\
& = (i) + (ii) + (iii) + (iv) + (v).
\end{align*}
The individual terms can now be estimated separately. 
By using first $\bar v_{1,h,\tau}=\dt(\rho-\hat\rho)$ as a test function in the variational identities \eqref{eq:pg3} and \eqref{eq:ppg3}, and then
$\bar w_{1,h,\tau}=\bar\mu_\rho-\hbmu_\rho-\bar r_2$ in \eqref{eq:pg1} and \eqref{eq:ppg1},
and recalling that we omit the subscript $h,\tau$ here, 
we see that 
\begin{align*}
(i) &= \la \bar\mu_\rho-\hbmu_\rho+\bar r_3,\dt(\rho-\hat\rho) \ra^n  \\
    &= - \la \bLL_{11}\nabla(\bar\mu_\rho-\hbmu_\rho),\nabla(\bar\mu_\rho-\hbmu_\rho-\bar r_3) \ra^n \\
    &\quad \;\! - \la (\bar\mu_{\eta}-\hbmu_{\eta})\bLL_{12},\nabla(\bar\mu_\rho-\hbmu_\rho-\bar r_{3}) \ra^n + \la \bar r_1,\bar\mu_\rho-\hbmu_\rho-\bar r_3 \ra^n.
\end{align*}
Here and below, $\bar \LL=\LL(\bar \omega)$ denotes the mobility matrix evaluated at the discrete solution $\omega=(\rho,\eta,\nabla \rho,\nabla \eta)$; the subscripts $h,\tau$ are omitted. 
%
%
With the same arguments, i.e. using $\bar w_{2,h,\tau}=\bar\mu_\eta-\hbmu_\eta-\bar r_4$ in  \eqref{eq:pg4} \& \eqref{eq:ppg4} as well as $\bar v_{2,h,\tau}=\dt(\eta-\hat\eta)$ in \eqref{eq:pg2} \& \eqref{eq:ppg2} we see that
\begin{align*}
(ii) &= - \la (\bar\mu_\eta-\hbmu_{\eta}+\bar r_{4})\bLL_{12},\nabla(\bar \mu_\rho-\hbmu_\rho) \ra^n \\
&\quad \;\! - \la \bLL_{22}(\bar\mu_{\eta}-\hbmu_{\eta}),\bar\mu_{\eta}-\hbmu_{\eta}+\bar r_4 \ra^n  + \la \bar r_{2},\bar\mu_{\eta}-\hbmu_{\eta} + r_{4} \ra^n.
\end{align*}
%
When combining the two expressions, we arrive at 
\begin{align*}
(i)+(ii) &= -\la \begin{pmatrix} \nabla(\bar\mu_{\rho}-\hbmu_{\rho}) \\ \bar\mu_{\eta}-\hbmu_{\eta}\end{pmatrix} + \begin{pmatrix} \nabla \bar r_{2} \\ \bar r_{4} \end{pmatrix},\bLL \begin{pmatrix} \nabla(\bar\mu_{\rho}-\hbmu_{\rho}) \\ \bar\mu_{\eta}-\hbmu_{\eta}\end{pmatrix}\ra^n \\
&\qquad \qquad \qquad  + \la \bar r_{1},\bar\mu_{\rho}-\hbmu_{\rho}-\bar r_{2} \ra^n + \la \bar r_{3},\bar\mu_{\eta}-\hbmu_{\eta} + \bar r_{4} \ra^n \\
&\le  \int\nolimits_{I_n} C(\LL,\delta) (\|\bar r_1\|_{H^{-1}}^2 + \|\bar r_2\|^2_{L^2} + \|\bar r_3\|^2_{H^1}) + \|\bar r_4\|_{L^2}^2 \,  \\
&\qquad \qquad -(1-\delta) \D_{\rho,\eta}(\bar\mu_{\rho}-\hbmu_{\rho},\bar\mu_{\eta}-\hbmu_{\eta}) + \la \bar\mu_{\rho}-\hbmu_{\rho},1\ra^{2}\, ds.
\end{align*}
The parameter $\delta$ stems from Young's inequalities and can be chosen as desired. 
The last term in the above estimate comes from the application of a Poincar\'e inequality and can be further bounded by
\begin{align}
\int_{I_n}\la \bar\mu_{\rho}-\hbmu_{\rho},1\ra^{2} \, ds
&= \int_{I_n}\la f_\rho(\rho,\eta) - f_\rho(\hat\rho,\hat\eta) - \bar r_{2}, 1\ra^{2} \, ds  \notag\\ 
& \leq \int_{I_n} c_2\E(\rho,\eta|\hat\rho,\hat\eta) + c_3\norm{\bar r_{2}}_{L^2}^2 \, ds\label{eq:mean_mu_ref}
\end{align}
In summary, the terms $(i)+(ii)$ can be estimated as required. 
Using the growth assumptions in (A3) and the available a-priori bounds for $\rho,\eta,\hat\rho,\hat\eta$ in $L^\infty(H^1)$, we get
\begin{align*}
(iii) + (iv)  
&= \la f_\rho(\rho,\eta|\hat\rho,\hat\eta),\dt\hat\rho \ra^n + \la f_\eta(\rho,\eta|\hat\rho,\hat\eta),\dt\hat\eta \ra^n  \\
& \leq \int_{I_n}\norm{D^2f_\rho(\xi_1,\xi_2)}_{L^6}\norm{\dt\hat\rho}_{L^2}\E_\alpha(\rho,\eta|\hat\rho,\hat\eta) \\
&\qquad \qquad + \norm{D^2f_\eta(\xi_1,\xi_2)}_{L^6}\norm{\dt\hat\eta}_{L^2}\E_\alpha(\rho,\eta|\hat\rho,\hat\eta) \, ds\\
&\leq c_f\int_{I_n}(\norm{\dt\hat\rho}_{L^2}+\norm{\dt\hat\eta}_{L^2})\E_\alpha(\rho,\eta|\hat\rho,\hat\eta) \, ds.
\end{align*}
The constant $c_f$ only depends on bounds in the assumptions and the initial data.
For the remaining term, we employ the variational identities \eqref{eq:pg1}--\eqref{eq:pg2} and \eqref{eq:ppg1}--\eqref{eq:ppg2} with test functions $\bar  v_{1,h,\tau}=\rho - \hat \rho$ and $\bar v_{2,h,\tau}=\eta - \hat \eta$, respectively, and obtain
\begin{align*}
(v) &= -\alpha\la \begin{pmatrix} \nabla(\bar\mu_{\rho}-\hbmu_{\rho}) \\ \bar\mu_{\eta}-\hbmu_{\eta}\end{pmatrix},\bLL \begin{pmatrix} \nabla(\rho-\hat\rho) \\ \eta-\hat\eta\end{pmatrix} \ra^n  + \alpha \la r_1,\rho-\hat\rho \ra^n + \alpha \la \bar r_{2},\eta-\hat\eta \ra^n \\
& \leq \int_{I_n} \delta \D_{\rho,\eta}(\mu_{\rho}-\hat\mu_{\rho},  \mu_{\eta}-\hat\mu_\eta +  c_5 \E_\alpha(\rho,\eta|\hat\rho,\hat\eta) + c_6 \norm{\bar r_{1}}_{H^{-1}}^2 + c_7 \norm{\bar r_{2}}_{L^2}^2 \, ds.
\end{align*}
By collecting all terms and choosing $\delta$ appropriately, we obtain the assertion.
\end{proof}

\section{Proof of Theorem~\ref{thm:fulldisk}}
\label{sec:err}

Without further mentioning, we assume that $(\rho,\eta,\mu_\rho,\mu_\eta)$ is a sufficiently smooth solution of \eqref{eq:1}--\eqref{eq:4} and denote by $(\rho_{h,\tau},\eta_{h,\tau},\mu_{\rho,h,\tau},\mu_{\eta,h,\tau})$ a corresponding approximation obtained by Problem~\ref{prob:pg} with initial values $\rho_{0,h}=\pi_h^1 \rho(0)$ and $\eta_{h,0}=\pi_h^1 \eta(0)$. 
%

\subsection{Error splitting and projection errors}

Using the spatial and temporal interpolation and projection operators, introduced in Section~\ref{sec:main}, we define the auxiliary functions
\begin{align} \label{eq:fullproj}
\hat \rho_{h,\tau} = \I_\tau^1 \pi_h^1 \rho, \qquad 
\hat \eta_{h,\tau} = \I_\tau^1 \pi_h^1 \eta, \qquad 
\hbmu_{\rho,h,\tau} = \bar \pi_\tau^0 \pi_h^0 \mu_\rho, \qquad 
\hbmu_{\eta,h,\tau} = \bar \pi_\tau^0 \pi_h^0 \mu_\eta.
\end{align}
In the usual manner, we then split the error $\omega-\omega_{h,\tau} = (\omega - \hat \omega_{h,\tau}) + (\hat \omega_{h,\tau}- \omega_{h,\tau})$ into projection and discrete error components, which may be estimated separately.
As a direct consequence of well-known approximation properties of these operators, see Appendix~\ref{asec:space}, one can deduce the following estimates.
\begin{lemma}[Projection error]\label{lem:projerr}
Let (A4)--(A6) hold. Then
\begin{align*}
&\|\hat\rho_{h,\tau} - \rho\|^2_{L^\infty(H^1)} \leq C(\tau^4 + h^4),& 
&\|\hbmu_\rho - \bar\mu_\rho\|^2_{L^2(H^1)} \leq Ch^4\\
&\|\hat\eta_{h,\tau} - \eta\|^2_{L^\infty(H^1)} \leq C(\tau^4 + h^4),& 
&\|\hbmu_\eta - \bar\mu_\eta\|^2_{L^2(L^2)} \leq Ch^4\\
&\|\dt \hat\rho_{h,\tau}-\bar \pi_\tau^0 (\dt \rho)\|^2_{L^2(H^{-1})} \leq C h^4,&
&\|\dt\hat\eta_{h,\tau}-\bar \pi_\tau^0 (\dt \eta)\|^2_{L^2(L^2)} \leq Ch^4.
\end{align*}
The constant $C$ in these estimates only depends on the bounds in the assumptions.
\end{lemma}

\subsection{Residuals}
In order to estimate the discrete error components $\hat \omega_{h,\tau} - \omega_{h,\tau}$, we will employ the stability results of the previous section. 
We start with inserting the projections of the exact solution into the discrete equations and determine the corresponding residuals. 

\begin{lemma} \label{lem:residuals}
Let $(\hat \rho_{h,\tau},\hat \eta_{h,\tau},\hbmu_\rho,\hbmu_\eta)$ be defined as in \eqref{eq:fullproj}. Then \eqref{eq:ppg1}--\eqref{eq:ppg4} holds with 
\begin{align*}
\la \bar r_{1,h,\tau}, \bar v_{1,h,\tau}\ra^n
&:= \la \dt (\pi_h^1 \rho - \rho), \bar v_{1,h,\tau}\ra^n + \la \bLL_{11,h,\tau}\nabla \hbmu_{\rho,h,\tau} - \LL_{11} \nabla \mu_\rho, \nabla \bar v_{1,h,\tau}\ra^n  
\\
& \qquad \qquad \qquad \qquad 
+ \la \bLL_{12,h,\tau}\hbmu_{\eta,h,\tau} - \LL_{12}\mu_\eta, \nabla \bar v_{1,h,\tau}\ra^n  , 
\\
\la \bar r_{2,h,\tau},\bar v_{2,h,\tau} \ra^n 
&:= \la \dt (\pi_h^1 \eta - \eta), \bar v_{2,h,\tau}\ra^n + \la \bLL_{12,h,\tau}\nabla \hbmu_{\rho,h,\tau} - \LL_{12} \nabla \mu_\rho, \bar v_{2,h,\tau}\ra^n 
\\
& \qquad \qquad \qquad \qquad 
+ \la \bLL_{22,h,\tau}\hbmu_{\eta,h,\tau} - \LL_{22}\mu_\eta, \bar v_{2,h,\tau}\ra^n, 
\\
\la  \bar r_{3,h,\tau}, \bar w_{1,h,\tau} \ra^n
&:= \la \hbmu_{\rho,h,\tau} - I_\tau^1 \mu_\rho, \bar w_{1,h,\tau}\ra^n 
      + \gamma \la \nabla (\hat \rho_{h,\tau} - \I_\tau^1 \rho), \nabla \bar w_{1,h,\tau}\ra^n 
\\
& \qquad \qquad \qquad \qquad 
+ \la f_\rho(\hat \rho_{h,\tau},\hat \eta_{h,\tau}) - \I_\tau^1 f_\rho(\rho,\eta), \bar w_{1,h,\tau}\ra^n,
\\
\la  \bar r_{4,h,\tau}, \bar w_{2,h,\tau} \ra^n
&:= \la \hbmu_{\eta,h,\tau} - I_\tau^1 \mu_\eta, \bar w_{2,h,\tau}\ra^n 
      + \gamma \la \nabla (\hat \eta_{h,\tau} - \I_\tau^1 \eta), \nabla \bar w_{2,h,\tau}\ra^n 
\\
& \qquad \qquad \qquad \qquad 
+ \la f_\eta(\hat \rho_{h,\tau},\hat \eta_{h,\tau}) - \I_\tau^1 f_\eta(\rho,\eta), \bar w_{2,h,\tau}\ra^n, 
\end{align*}
As before, $\bar \LL_{h,\tau}=\LL(\bar \omega_{h,\tau})$ and $\LL=\LL(\omega)$ denote the mobility matrices evaluated at the continuous and discrete solutions $\omega_{h,\tau}=(\rho_{h,\tau},\eta_{h,\tau},\nabla \rho_{h,\tau},\nabla \eta_{h,\tau})$ and $\omega=(\rho,\eta,\nabla \rho,\nabla \eta)$.
\end{lemma}
\begin{proof}
The result follows immediately by plugging the projections into the discrete variational problem \eqref{eq:pg1}--\eqref{eq:pg4} and using some elementary properties of the projections as well as the variational characterization \eqref{eq:1}--\eqref{eq:4} of the true solution. 
\end{proof}

\subsection{Bounds for the residuals}
As a next step, we derive quantitative estimates for the residuals in terms of the discretization parameters.

\begin{lemma}\label{lem:res_est}
Let (A0)--(A6) hold and the residuals $\bar r_{i,h,\tau}$ be defined as in Lemma~\ref{lem:residuals}. Then 
\begin{align*}
 \int_{I_n}  \norm{\bar r_{1,h,\tau}}_{H^{-1}}^2 &+ \norm{\bar r_{2,h,\tau}}_{H^1}^2  + \norm{\bar r_{3,h,\tau}}_{L^2}^2 + \norm{\bar r_{4,h,\tau}}_{L^2}^2  \, ds  \\
 & \leq c\int_{I_n}\E_\alpha(\rho_{h,\tau},\eta_{h,\tau}|\hat\rho_{h,\tau},\hat\eta_{h,\tau}) \, ds + \hat C  (h^4 + \tau^4)
\end{align*}
with constants $c,\hat C$ depending only on the bounds in the assumptions. 
\end{lemma}
\begin{proof}
The proof is rather technical and thus split into several steps. For abbreviation, we introduce $\hbLL:=\LL(\hbrho,\hbeta,\nabla\hbrho,\nabla\hbeta)$ as shorthand notation for the evaluation of the mobility matrix at the projected solution components.
%

\subsubsection*{Estimates for the mobilities.}
Under the assumptions of Theorem~\ref{thm:fulldisk}, one has 
\begin{align*}
%
\int_{I_n}\norm{\bLL_{h,\tau}-\bLL}_{L^2}^2 \, ds &\leq C'\int_{I_n}\E_\alpha(\rho_{h,\tau},\eta_{h,\tau}|\hat\rho_{h,\tau},\hat\eta_{h,\tau}) \, ds  + C''(h^4 + \tau^4),
\end{align*} 
with constants $C'$, $C''$ independent of $h$ and $\tau$.
To see this, we split the error by 
$$
\int_{I_n} \norm{\bLL_{h,\tau}-\hbLL_{h,\tau}}_{L^2}^2 \, ds 
\le \int_{I_n} 2 \norm{\bLL_{h,\tau}-\hbLL_{h,\tau}}_{L^2}^2 + 2 \norm{\hbLL_{h,\tau}-\hbLL}_{L^2}^2 \, ds,
$$
and then estimate the two terms separately. By assumption (A2), we obtain 
\begin{align*}
   |\bLL_{h,\tau}-\hbLL_{h,\tau}|&=|\LL(\bar\rho_{h,\tau},\bar\eta_{h,\tau},\nabla\bar\rho_{h,\tau},\nabla\bar\eta_{h,\tau})-\LL(\hbrho_{h,\tau},\hbeta_{h,\tau},\nabla\hbrho_{h,\tau},\nabla\hbeta_{h,\tau})|, \\
   &\leq C (|\bar\rho_{h,\tau}-\hbrho_{h,\tau}|  + |\bar\eta_{h,\tau}-\hbeta_{h,\tau}| + |\nabla(\bar\rho_{h,\tau}-\hbrho_{h,\tau})| + |\nabla(\bar\eta_{h,\tau}-\hbeta_{h,\tau})|),
\end{align*}
which follows immediately from Taylor estimates and bounds on the derivatives.
The bound for the first term in the above splitting then follows immediately from the lower bound for the relative energy stated in Lemma~\ref{lem:properties}.
For the second term, we use
\begin{align*}
    \int_{I_n}\norm{\hbLL_{h,\tau}-\bLL}_{L^2}^2 \, ds
    &= \int_{I_n}\norm{\LL(\hbrho_{h,\tau},\hbeta_{h,\tau},\nabla\hbrho_{h,\tau},\nabla\hbeta_{h,\tau})-\overline{\LL(\rho,\eta,\nabla\rho,\nabla\eta)}}_{L^2}^2 \, ds\\
    & \leq \int_{I_n}\norm{\LL(\hbrho_{h,\tau},\hbeta_{h,\tau},\nabla\hbrho_{h,\tau},\nabla\hbeta_{h,\tau})-\LL(\bar\rho,\bar\eta,\nabla\bar\rho,\nabla\bar\eta)}_{L^2}^2 \\
    &\qquad + \norm{\LL(\bar\rho,\bar\eta,\nabla\bar\rho,\nabla\bar\eta)-\overline{\LL(\rho,\eta,\nabla\rho,\nabla\eta)}}_{L^2}^2 \, ds  =: (a) + (b).
\end{align*}
Let us recall that $\bar \omega = \bar \pi_\tau^0 \omega$ is used to denote the piecewise constant projection in time.
By appropriately expanding the first term and $(\sum_{i=1}^N a_i)^2 \le N \sum_{i=1}^N a_i^2$, we can see that 
\begin{align*}
  (a) \le \changestwo{4} \int_{I_n}&\norm{\LL(\hbrho_{h,\tau},\hbeta_{h,\tau},\nabla\hbrho_{h,\tau},\nabla\hbeta_{h,\tau})-\LL(\bar\rho,\hbeta_{h,\tau},\nabla\hbrho_{h,\tau},\nabla\hbeta_{h,\tau})}_{L^2}^2 \\
  &+ \norm{\LL(\bar\rho,\hbeta_{h,\tau},\nabla\hbrho_{h,\tau},\nabla\hbeta_{h,\tau}) - \LL(\bar\rho,\hbeta_{h,\tau},\nabla\bar\rho,\nabla\hbeta_{h,\tau})}_{L^2}^2 \\
  &+ \norm{\LL(\bar\rho_{h,\tau},\hbeta_{h,\tau},\nabla\bar\rho,\nabla\hbeta_{h,\tau}) - \LL(\bar\rho,\bar\eta,\nabla\bar\rho,\nabla\hbeta_{h,\tau})}_{L^2}^2 \\
 &+ \norm{\LL(\bar\rho,\bar\eta,\nabla\bar\rho,\nabla\hbeta_{h,\tau}) - \LL(\bar\rho,\bar\eta,\nabla\bar\rho,\nabla\bar\eta)}_{L^2}^2 \, ds
\end{align*}
By multiple applications of Taylor estimates, the uniform bounds on $\LL$ and its derivatives in assumption (A2), and the estimates of Lemma~\ref{lem:projerr} for the projection errors, we then get
\begin{align*}
 (a) 
 \leq C(\LL)\int_{I_n} \norm{\hbrho_{h,\tau}-\bar\rho}_1^2  + \norm{\hbeta_{h,\tau}-\bar\eta}_1^2 \, ds 
 \leq C(\norm{\rho}_{L^2(H^3)}^2+\norm{\rho}_{L^2(H^3)}^2) \, h^4.    
\end{align*}
For the second term, we use the estimate \eqref{eq:midpoint_order_single} to see that 
\begin{align*}
  (b) 
  &= \int_{I_n}\norm{\overline{\LL(\bar\rho,\bar\eta,\nabla\bar\rho,\nabla\bar\eta)}-\overline{\LL(\rho,\eta,\nabla\rho,\nabla\eta)}}_{L^2}^2 \, ds \\
   &= C^* \tau^4 \norm{\LL(\rho,\eta,\nabla\rho,\nabla\eta)}_{H^2(L^2)}^2 
   \leq C\tau^4(\norm{\rho}_{H^2(H^2)}^2 + \norm{\eta}_{H^2(H^2)}^2).
\end{align*}
The constant $C^*$ here depends on the bounds for $\LL$ and its derivatives as well as on lower order norms of $\rho$ and $\eta$.
By combination of the previous results, we thus obtain the stated bounds for the differences in the mobilities. 
We now turn to the estimates for the residuals.

\subsubsection*{First residual}
Using the triangle inequality, we immediately get 
\begin{align*}
\int_{I_n}  \norm{\bar r_{1,h,\tau}}_{H^{-1}}^2 \, ds
&\leq 3 \big( \norm{\dt(\pi_h^1\rho-\rho)}_{L^2(H^{-1})}^2 + \norm{\bLL_{11}\nabla\hbmu_{\rho,h,\tau}-\overline{\LL_{11}\nabla\mu_\rho}}^2_{L^2(L^2)}\\ 
& \qquad \qquad \qquad \qquad 
+ \norm{\bLL_{12}\hbmu_{\eta,h,\tau}-\overline{\LL_{12}\mu_\eta}}^2_{L^2(L^2)} \big) =: (i) + (ii) + (iii).
\end{align*}
Via Lemma~\ref{lem:projerr}, the first term can be estimated by $(i) \leq C h^4\norm{\dt\rho}_{L^2(H^1)}^2$.
The second term can again be expanded into several parts by 
\begin{align*}
   (ii) & = 3 \big( \norm{(\bLL_{11,h,\tau}-\bLL_{11})\nabla\hbmu_{\rho,h,\tau}}_{L^2(L^2)}^2  + \norm{\bLL_{11}\nabla(\hbmu_{\rho,h,\tau}-\bmu_\rho)}_{L^2(L^2)}^2 \\
   &\qquad \qquad \qquad \qquad 
   + \norm{\bLL_{11}\nabla\bmu_\rho-\overline{\LL_{11}\nabla\mu_\rho}}_{L^2(L^2)}^2 \big)
   = :(iia) + (iib) + (iic),
\end{align*}
which can now readily be estimated using Lemma~\ref{lem:projerr}, Lemma~\ref{lem:average_time_err}, the bounds for the mobility errors derived above, and Hölder inequalities. In this manner, we obtain
\begin{align*}
   (iia) &\leq  \norm{\nabla\hbmu_{\rho,h,\tau}}_{L^\infty(L^\infty)}^2 \int_{I_n}  \norm{\bLL_{11,h,\tau}-\hbLL_{11,h,\tau}}_0^2 + \norm{\hbLL_{11,h,\tau}-\bLL_{11}}_0^2 \, ds \\
& \leq C\int_{I_n}  \E_\alpha(\rho_{h,\tau},\eta_{h,\tau}|\hat\rho_{h,\tau},\hat\eta_{h,\tau}) \, ds + C(h^4 + \tau^4), 
\end{align*}
as well as $(iib) \leq Ch^4\norm{\mu_\rho}_{L^2(H^3)}^2$ and $(iic) \leq C\tau^4\norm{\LL_{11}\nabla\mu_\rho}_{H^2(H^1)}^2$.
Together this yields
\begin{align*}
    (ii) \leq C\norm{\nabla\hbmu_{\rho,h,\tau}}_{L^\infty(L^\infty)}^2 \int_{I_n}  \E_\alpha(\rho_{h,\tau},\eta_{h,\tau}|\hat\rho_{h,\tau},\hat\eta_{h,\tau}) \, ds + C(h^4+\tau^4)
\end{align*}
The third term in the previous splitting can be treated similarly, leading to
\begin{align*}
   (iii) \leq  C\norm{\hbmu_{\eta,h,\tau}}_{L^\infty(L^\infty)}^2 \int_{I_n}  \E_\alpha(\rho_{h,\tau},\eta_{h,\tau}|\hat\rho_{h,\tau},\hat\eta_{h,\tau}) \, ds + C(h^4+\tau^4).
\end{align*}
Using stability properties of the $L^2$-projection in space and time, cf \eqref{eq:l2stabw1p}, we find that 
\begin{align*}
 \norm{\nabla\hbmu_{\rho,h,\tau}}_{L^\infty(L^\infty)}^2  \leq C\norm{\nabla\mu_\rho}_{L^\infty(L^\infty)}, 
 \qquad  
 \norm{\hbmu_{\eta,h,\tau}}_{L^\infty(L^\infty)}^2  \leq C\norm{\mu_\eta}_{L^\infty(L^\infty)}.
\end{align*}
Putting everything together, we have thus shown that
\begin{align*}
    \int_{I_n}  \norm{\bar r_{1,h,\tau}}_{H^1}^2 \, ds&\leq \tilde C_1\int_{I_n}  \E_\alpha(\rho_{h,\tau},\eta_{h,\tau}|\hat\rho_{h,\tau},\hat\eta_{h,\tau}) \, ds + C_1'(h^4+\tau^4)
\end{align*}
with constants $C_1,C_1'$ that are independent of $h$ and $\tau$.

\subsubsection*{Second residual.}
With exactly the same arguments as used above for the estimation of the first residual, we obtain the bound
\begin{align*}
    \int_{I_n}  \norm{\bar r_{2,h,\tau}}_{L^2}^2 \, ds &\leq \tilde C_2'\int_{I_n}  \E_\alpha(\rho_{h,\tau},\eta_{h,\tau}|\hat\rho_{h,\tau},\hat\eta_{h,\tau}) \, ds + C_2''(h^4+\tau^4).
\end{align*}
The constants $C_2'$, $C_2''$ are again independent of $h$ and $\tau$.

\subsubsection*{Third residual.}
Let us start with the observation that
\begin{align} \label{eq:aux2}
\la \nabla (\hat \rho_{h,\tau} - \I_\tau^1 \rho), \nabla \bar \xi_{h,\tau} \ra^n
&= \la \I_\tau^1 \phi - \hat \rho_{h,\tau}, \bar \xi_{h,\tau}\ra^n,
\end{align}
which follows immediately from the definition of $\hat \rho_{h,\tau}$ and \eqref{eq:defh1proj}.
Hence, we can estimate
\begin{align*}
\int_{I_n}  \norm{\bar r_{3,h,\tau}}_{H^1}^2 \, ds &\leq 3 \big( \norm{\pi_h^0\mu_{\rho}- I_1\pi_h^0\mu_\rho}_{L^2(H^{1})}^2 + \norm{\hat\rho_{h,\tau}-I_1\rho}^2_{L^2(H^1)} \\
& \qquad \qquad  +
\norm{f_\rho(\hat\rho_{h,\tau},\hat\eta_{h,\tau}) - I^1_\tau f_\rho(\hat\rho,\hat\eta) }^2_{L^2(H^1)} \big) 
=(i) + (ii) + (iii).
\end{align*}
By Lemma~\ref{lem:projerr}, the first term can be bounded by $(i) \leq C\tau^4\norm{\mu_\rho}_{H^2(H^1)}$. 
The second term can be split into several parts and then estimated by
\begin{align*}
(ii) 
&\le 3 \big(\|\I_\tau^1 \rho - \rho\|_{L^2(H^1)}^2 + \|\rho - \pi_h^1\rho\|^2_{L^2(H^1)} + \|\pi_h^1 \rho - \I_\tau^1 \pi_h^1 \rho\|_{L^2(H^1)}^2 \big) \\
&\le C  \big( h^4 \|\rho\|_{L^2(H^3)}^2 + \tau^4 \|\rho\|_{H^2(H^1)}^2 \big).
\end{align*}
From assumption (A7) and the properties of the projection operators, we conclude that 
$\rho,\eta$ and their projections $\hat\rho_{h,\tau},\hat\eta_{h,\tau}$ can be bounded uniformly in $L^\infty(W^{1,\infty})$.
Therefore, all derivatives of $f(\cdot,\cdot)$ appearing in the following can be bounded by a constant $C_f$.
As a consequence, we then obtain 
\begin{align*}
(iii) 
&\le 2 \big( \|f_\rho(\hat \rho_{h,\tau},\hat\eta_{h,\tau}) - f_\rho(\rho,\eta)\|_{L^2(H^1)}^2 + \|f_\rho(\rho,\eta) - \I_\tau^1 f_\rho(\rho,\eta)\|_{L^2(H^1)}^2 \big) 
\\
&\le C_f \big( \|\hat\rho_{h,\tau} - \rho\|_{L^2(H^1)}^2 + \|\hat\eta_{h,\tau} - \eta\|_{L^2(H^1)}^2\big) + C \tau^4\|f_\rho(\rho,\eta)\|_{H^2(H^1)}^2.
\end{align*}
The last term involves higher derivatives of $f$ as well as cubic products of $\rho,\eta$ and its derivatives, with the highest order terms are given by $\dtt\nabla\rho$ and $\dtt\nabla\eta$. This yields to
\begin{align*}
 \|f_\rho(\rho,\eta)\|_{H^2(H^1)} \leq C_f (1 + \norm{(\rho,\eta)}_{H^2(H^1)} + \norm{(\rho,\eta)}_{H^1(H^3)})^3  
\end{align*}
In summary, the second residual may thus be estimated by
\begin{align*}
\int_{I_n} \|\bar r_{3,h,\tau}\|_{H^1}^2 \,ds 
\le C_3(h^4+\tau^4)
\end{align*}
with constant $C_2$ independent of the discretization parameters $h$ and $\tau$.

\subsubsection*{Fourth residual.}
The bound for the fourth residual is again obtain with exactly the same arguments as used for the estimation of the third residual. Here we obtain
\begin{align*}
\int_{I_n} \|\bar r_{4,h,\tau}\|_{L^2}^2 \,ds 
&\le C_4 (h^4+\tau^4)
\end{align*}
with constant $C_4$ independent of $h$ and $\tau$. This concludes the proof of the lemma.
\end{proof}

\subsection{Estimate for the discrete error}
A combination of the previous results allows us to prove the following estimate for the discrete error component. 
\begin{lemma}\label{lem:dic_err}
Let (A0)--(A6) hold. Then 
\begin{align*}
 \norm{\rho_{h,\tau}&-\hat\rho_{h,\tau}}_{L^\infty(H^1)} + \norm{\eta_{h,\tau}-\hat\eta_{h,\tau}}_{L^\infty(H^1)} \\
  &+ \norm{\bmu_{\rho,h,\tau}-
      \hbmu_{\rho,h,\tau}}_{L^2(H^1)} + \norm{\bmu_{\eta,h,\tau}-\hbmu_{\eta,h,\tau}}_{L^2(L^2)} \leq C(h^2+\tau^2)
  \end{align*}
\end{lemma}
\begin{proof}
We use the bounds for the relative energy in Lemma~\ref{lem:properties} and the fact that the values $u(t) = \lambda u^n + (1-\lambda) u^{n-1}$ of a linear function on $I_n$ is a convex combination of its terminal values. Together with the convexity of the squared $L^2$-norm, we then get 
\begin{align*}
\int_{I_n} \E_\alpha(u|\hat u) \, ds
&\le C_\alpha \int_{I_n} \|u - \hat u\|^2 \, ds\\
&\le C_\alpha \tau \int_0^1 \lambda \|u^n - \hat u^n\|^2 + (1-\lambda) \|u^{n-1} - \hat u^{n-1}\|^2 d\lambda \\
&\le \frac{C_\alpha}{cc_\alpha} \tau (\E_\alpha(u^n|\hat u^n) + \E_\alpha(u^{n-1}|\hat u^{n-1})). 
\end{align*}
By combination of Theorem~\ref{thm:stab}, Lemma~\ref{lem:residuals}, and Lemma~\ref{lem:res_est}, we then obtain 
\begin{align*}
\E_\alpha(\rho_{h,\tau},\eta_{h,\tau}&|\hat\rho_{h,\tau},\hat\eta_{h,\tau})\Big\vert_{t^{n-1}}^{t^n} + C_1 \int_{I_n} \D_{\bar\rho_{h,\tau},\bar\eta_{h,\tau}}(\mu_{\rho,h,\tau}-\hat\mu_{\rho,h,\tau},\mu_{\eta,h,\tau}-\hat\mu_{\eta,h,\tau}) \, ds \\
&\leq C_2 \tau \left(\E_\alpha(\rho_{h,\tau}^n,\eta_{h,\tau}^n|\hat\rho_{h,\tau}^n,\hat\eta_{h,\tau}^n) + \E_\alpha(\rho_{h,\tau}^{n-1},\eta_{h,\tau}^{n-1}|\hat\rho_{h,\tau}^{n-1},\hat\eta_{h,\tau}^{n-1}) \right) 
+ C_3 (h^4 + \tau^4).
\end{align*}
After a slight rearrangement of terms and using $(1- C_2 \tau) \ge e^{-\lambda \tau}$ and $(1+C_2 \tau) \le e^{\lambda \tau}$ for all $\tau \le 1/(2C_2)$ with some constant $\lambda \approx C_2$, this can be written in compact form  
\begin{align*}
e^{-\lambda \tau} u^n &+ b^n \le e^{\lambda \tau} u^{n-1} + d^n
\end{align*}
with
$u^n =\E_\alpha(\rho_{h,\tau}^n,\eta_{h,\tau}^n|\hat\rho_{h,\tau}^n,\hat\eta_{h,\tau}^n)$,
$b^n =  \frac{1}{2} \int_{I_n} \D_{\bar \rho_{h,\tau},\bar \eta_{h,\tau}}(\bmu_{\rho,h,\tau}-\hbmu_{\rho,h,\tau}, \bmu_{\eta,h,\tau}-\hbmu_{\eta,h,\tau}) \, ds$, 
and 
$d^n =  \bar C \hat C (h^2 + \tau^2)^2$.
By the discrete Gronwall inequality \eqref{eq:discgronwall}, we thus obtain 
\begin{align*}
\E_\alpha(\rho_{h,\tau}^n,\eta_{h,\tau}^n|\hat\rho_{h,\tau}^n,\hat\eta_{h,\tau}^n) &+ \int_0^{t^n} \D_{\bar \rho_{h,\tau},\bar \eta_{h,\tau}}(\bmu_{\rho,h,\tau}-\hbmu_{\rho,h,\tau}, \bmu_{\eta,h,\tau}-\hbmu_{\eta,h,\tau}) \, ds \\
&\le C_T \E_\alpha(\rho_{h,\tau}^0,\eta_{h,\tau}^0|\hat\rho_{h,\tau}^0,\hat\eta_{h,\tau}^0) + C_T' (h^4 + \tau^4),
\end{align*}
with constants $C_T$ and $C_T'$ depending only on the maximal time $T$ and the previous bounds. 
Due to the choice of the initial values in Problem~\ref{prob:pg}, the first term on the right hand side vanishes identically.
Lemma~\ref{lem:properties} then allows us to estimate the relative energy from below by the norm. 
With the lower bounds in (A1)--(A3) and the estimate \eqref{eq:mean_mu_ref}, we can further estimate the dissipation term from below by the corresponding norms. 
In summary, this yields the estimates stated in the lemma. 
\end{proof}

\subsection{Convergence rates.}
The convergence rates announced in Theorem~\ref{thm:fulldisk} now immediately follow by splitting the error into projection and discrete error components and application of Lemma~\ref{lem:projerr} and Lemma~\ref{lem:dic_err} to estimate the individual components. 
\hfill 
\qed

\subsection{Discrete uniqueness}
Let $(\rho_{h,\tau},\bmu_{\rho,h,\tau},\eta_{h,\tau},\bmu_{\eta,h,\tau})$ be a solution of Problem~\ref{prob:pg} and and $(\hat\rho_{h,\tau},\hbmu_{\rho,h,\tau},\hat\eta_{h,\tau},\hbmu_{\eta,h,\tau})$ be another solution with the same initial values.  
As before, we may understand the latter as a solution of the perturbed system \eqref{eq:ppg1}--\eqref{eq:ppg2} which is linearized around the original solution. 
The corresponding residuals are 
\begin{align}
\la \bar r_{1,h,\tau}, \bar v_{1,h,\tau}\ra^n
&=  \la (\bLL_{11,h,\tau}- \hbLL_{11,h,\tau})\nabla \hbmu_{\rho,h,\tau}   + (\bLL_{12,h,\tau}- \hbLL_{12,h,\tau})\hbmu_{\eta,h,\tau} , \nabla \bar v_{1,h,\tau}\ra^n  
\\
\la \bar r_{2,h,\tau},\bar v_{2,h,\tau} \ra^n 
&=  \la (\bLL_{12,h,\tau} - \hbLL_{12,h,\tau})\nabla \hbmu_{\rho,h,\tau} + (\bLL_{22,h,\tau} - \hbLL_{22,h,\tau})\hbmu_{\eta,h,\tau}, \bar v_{2,h,\tau}\ra^n,
\end{align}
and $\bar r_{3,h,\tau} = \bar r_{4,h,\tau}=0$, which can be verified by elementary considerations. 
With similar arguments as used in the previous section, we obtain the following bounds. 
\begin{lemma} \label{lem_est_res_uniq}
Under the previous assumptions, we have   
\begin{align*}
      \int_{I_n}\norm{\bar r_{1,h,\tau}}_{H^{-1}}^2 +  \norm{\bar r_{3,h,\tau}}_{L^2}^2 \, ds \leq C \int_{I_n} \E_\alpha(\rho_{h,\tau},\eta_{h,\tau}|\hat\rho_{h,\tau},\hat\eta_{h,\tau}) \, ds.
   \end{align*}
\end{lemma}
\begin{proof}
With the very same arguments as employed in the proof of Lemma~\ref{lem:res_est}, we see that
\begin{align*}
    &\int_{I_n} \norm{r_{1,h,\tau}}_{H^{-1}}^2 + \norm{r_{2,h,\tau}}_{L^2}^2 \, ds \\
    & \leq C' \, (\norm{\nabla\hbmu_{\rho,h,\tau}}^2_{L^\infty(L^\infty)}+\norm{\hbmu_{\eta,h,\tau}}^2_{L^\infty(L^\infty)})\int_{I_n} \E_{\alpha}(\rho_{h,\tau},\eta_{h,\tau}|\hat\rho_{h,\tau},\hat\eta_{h,\tau}) \, ds.
\end{align*}
In order to verify that $\norm{\nabla\hbmu_{\rho,h,\tau}}^2_{L^\infty(L^\infty)}$ and $\norm{\hbmu_{\eta,h,\tau}}^2_{L^\infty(L^\infty)}$ can be bounded independently of $h,\tau$, we make use of the true solution and the previous error estimates, and compute
\begin{align*}
 \norm{\nabla\hbmu_{\rho,h,\tau}}^2_{L^\infty(L^\infty)} \leq \norm{\hbmu_{\rho,h,\tau}&-\pi_0^\tau\pi_0^h\mu_\rho}^2_{L^\infty(W^{1,\infty})} \\
 & + \norm{\pi_0^\tau\pi_0^h\mu_\rho - \pi_0^\tau\mu_\rho}^2_{L^\infty(W^{1,\infty})} +   \norm{ \pi_0^\tau\mu_\rho}^2_{L^\infty(W^{1,\infty})}.
\end{align*}
Using error estimates for discrete solutions and an inverse inequality, we see that
\begin{align*}
 \norm{\hbmu_{\rho,h,\tau}-\pi_0^\tau\pi_0^h\mu_\rho}^2_{L^\infty(W^{1,\infty})} &\leq \tau^{-1}h^{-3}\norm{\hbmu_{\rho,h,\tau}-\pi_0^\tau\pi_0^h\mu_\rho}^2_{L^2(H^{1})} 
 \leq C(\tau^3 h^{-3} + \tau^4 h) \le C',
\end{align*}
where we used $\tau \lesssim h$ in the last step. The remaining two terms and $\norm{\hbmu_{\eta,h,\tau}}^2_{L^\infty(L^\infty)}$can be estimated in a similar manner. 
\end{proof}
%
By proceeding with essentially the same arguments as already used in the proof of the convergence rates result, we then obtain 
\begin{align*}
\E_\alpha(\rho_{h,\tau}^n,\eta_{h,\tau}^n|\hat\rho_{h,\tau}^n,\hat\eta_{h,\tau}^n) 
\le C''_T \E_\alpha(\rho_{h,\tau}^0,\eta_{h,\tau}^0|\hat\rho_{h,\tau}^0,\hat\eta_{h,\tau}^0).
\end{align*}
Since the two discrete solutions were assumed to have the same initial values, the right hand side vanishes. By  Lemma~\ref{lem:properties}, we conclude uniqueness of the discrete solution.
\hfill 
\qed

{

\section{Numerical tests}\label{sec:num}
%
For illustration of the viability of the proposed method and illustration of our theoretical findings, we now present some computational results and report about the convergence rates for a typical test problem.

\subsection*{Model problem}

We consider the domain $\Omega=(0,1)^2$and complement \eqref{eq:s1}--\eqref{eq:s2} is by periodic boundary conditions. We further introduce $\mathbf{n}_{c}(\rho):=\nabla\rho/\sqrt{c+\snorm{\nabla\rho}^2}$, which for $c>0$ is a regularized normal to the isolines of the phase field $\rho$.
We then define
\begin{align*}
 \LL_{11} &= \mathbf{I} + \frac{1}{\LL_{22}}\LL_{12}\LL_{12}^\top =  \mathbf{I} + \mathbf{n}_{1}(\rho)\otimes \mathbf{n}_{1}(\rho), \LL_{12} = \sqrt{1000}\mathbf{n}_{1}(\rho), \LL_{22} = 1000, \\
 f(\rho,\eta) &= C\rho^2(1-\rho)^2 + D[\rho^2 + 6(1-\rho)(\eta^2 + (1-\eta)^2) \\
 &- 4(2-\rho)(\eta^3 + (1-\eta)^3)  + 3(\eta^2 + (1-\eta)^2 )^2 ]  \\  
 C&= 1, \quad D=0.062, \gamma_\rho=\gamma_\eta=10^{-3}
\end{align*}

The choice for $f(\cdot,\cdot)$ is typical in the literature; cf.~\cite{Tonks_2015,oyedeji2022}. 
We further set $T=0.1$ for the final time and choose the following periodic but non-symmetric initial conditions
\begin{equation*}
 \rho_0  =  0.5 + 0.5\sin(2\pi x)\sin(2\pi y) ,\quad \eta_0  =  0.5 + 0.5\sin(4\pi x)\sin(2\pi y).
\end{equation*}

For all our computations, we use the method introduced in Problem~\ref{prob:pg} with spatial and temporal mesh size $h$ and $\tau$. In order to guarantee existence of a unique discrete solution, see Theorem~\ref{thm:fulldisk}, we choose $\tau = c h$ in all our tests with $c=0.001$. 
The choice of the time step is motivated from independent evaluations of the spatial and temporal discretization errors and leads to an approximate balance between them.
For the solution of the nonlinear system \eqref{eq:pg1}--\eqref{eq:pg4} in every time step, we employ the Newton method, which in all our tests converged within $6$ iterations to almost machine precision. 

In Figures~\ref{fig:evorho} and \ref{fig:evoeta}, we depict some snapshots of the phase fields $\rho$ and $\eta$ obtained in our simulations.
\begin{figure}[htbp!]
\centering
\footnotesize
\begin{tabular}{ccc}
\hspace{-1.5em}t=0.0156 & \hspace{-1.5em}t=0.0046 & \hspace{-1.5em}t=0.0093 \\
    \includegraphics[trim={2.0cm 0.0cm 0.0cm 0.5cm},clip,scale=0.32]{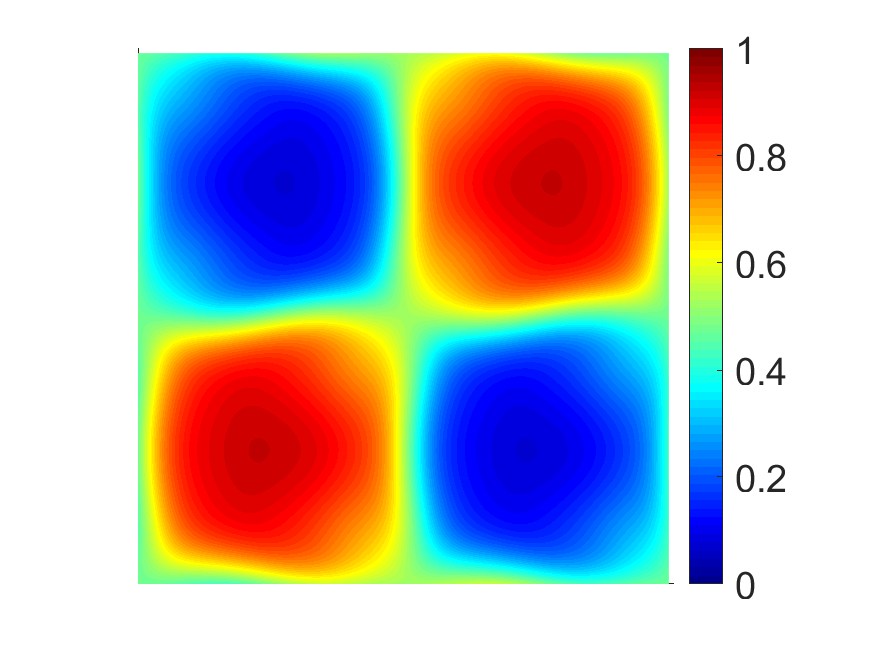} 
    &
    \includegraphics[trim={2.0cm 0.0cm 0.0cm 0.5cm},clip,scale=0.32]{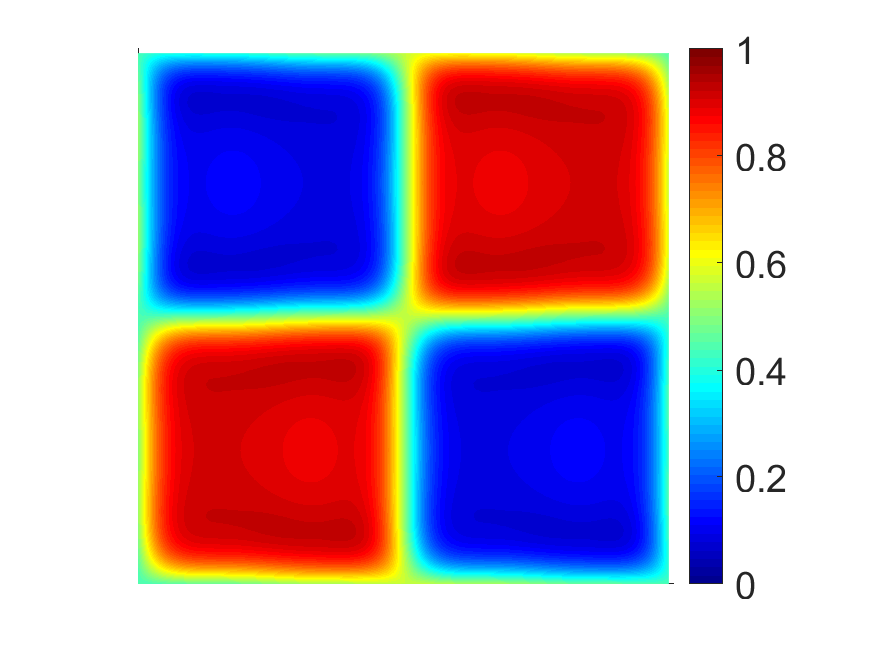}  
    &
    \includegraphics[trim={2.0cm 0.0cm 0.0cm 0.5cm},clip,scale=0.32]{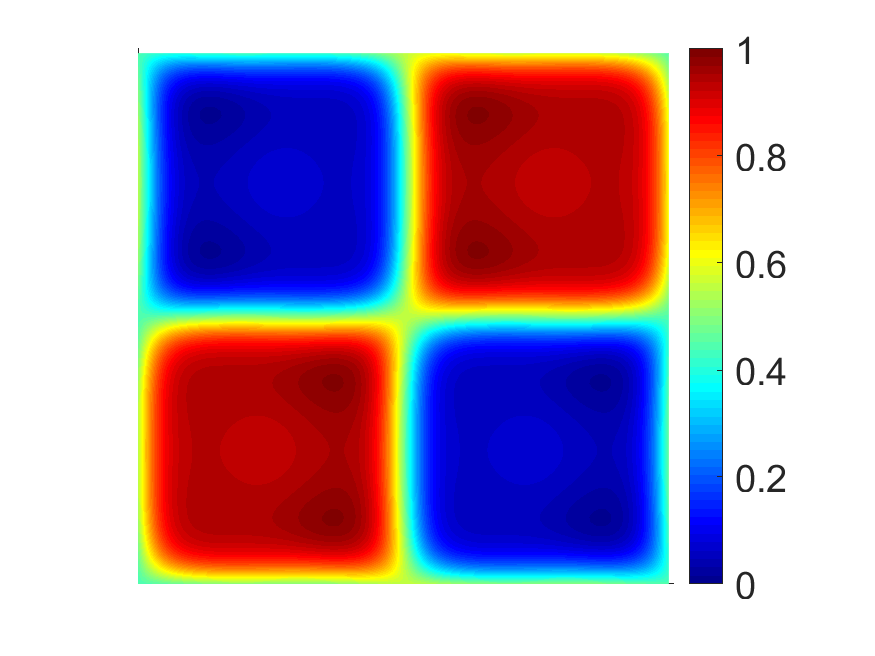}  \\[-0.5em]
    \hspace{-1.5em}t=0.0312 & \hspace{-1.5em}t=0.0625 & \hspace{-1.5em}t=0.1 \\
    \includegraphics[trim={2.0cm 0.0cm 0.0cm 0.5cm},clip,scale=0.32]{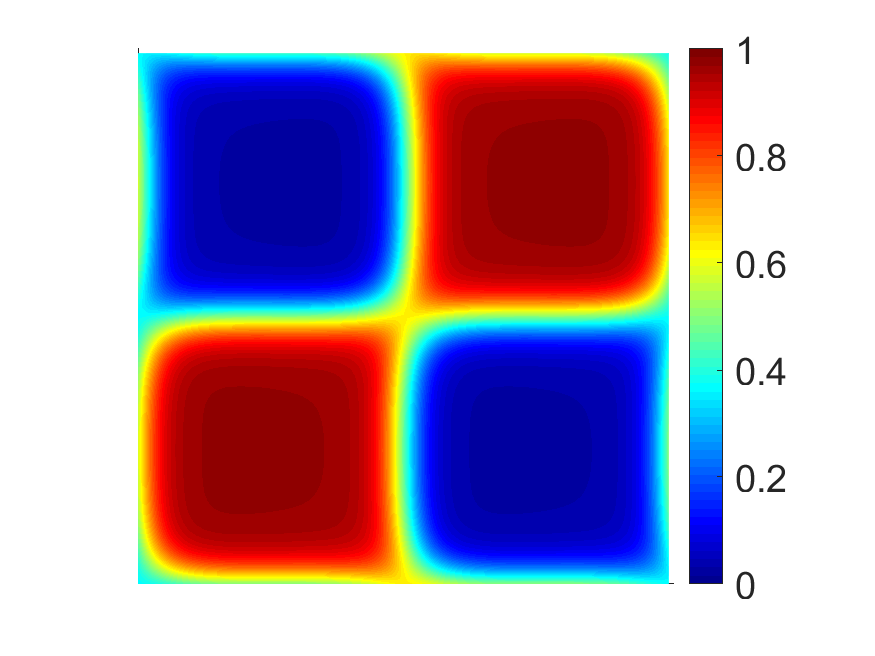}  
    &
    \includegraphics[trim={2.0cm 0.0cm 0.0cm 0.5cm},clip,scale=0.32]{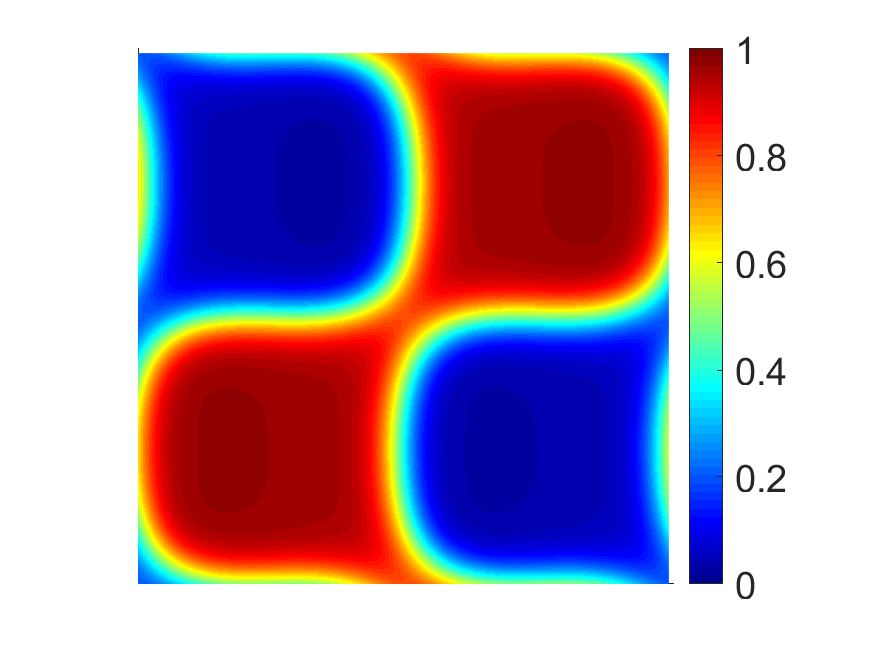} 
    &
    \includegraphics[trim={2.0cm 0.0cm 0.0cm 0.5cm},clip,scale=0.32]{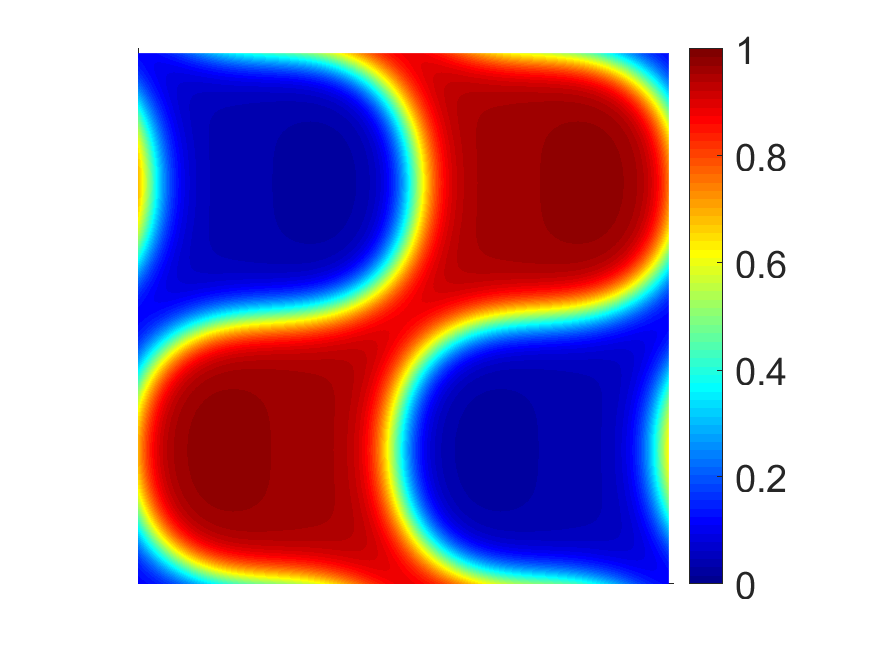} \\[-0.5em]
\end{tabular}
    \caption{Snapshots of the conserved phase field $\rho$ for simulations with $h=2^{-6}$ and $\tau=0.001\cdot h$.\label{fig:evorho}}
\end{figure}
\begin{figure}[htbp!]
\centering
\footnotesize
\begin{tabular}{ccc}
\hspace{-1.5em}t=0.0156 & \hspace{-1.5em}t=0.0046 & \hspace{-1.5em}t=0.0093 \\
    \includegraphics[trim={2.0cm 0.0cm 0.0cm 0.5cm},clip,scale=0.32]{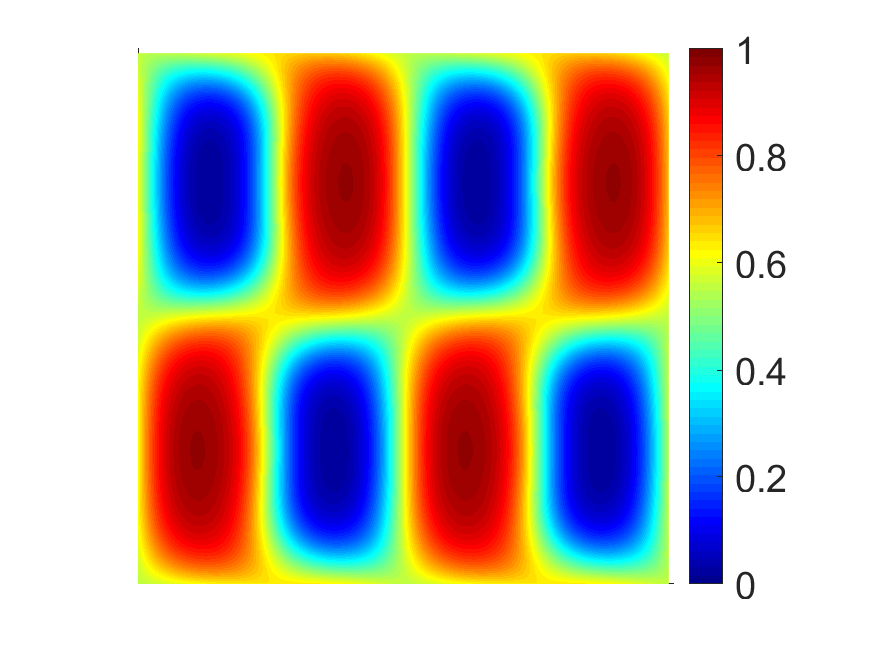} 
    &
    \includegraphics[trim={2.0cm 0.0cm 0.0cm 0.5cm},clip,scale=0.32]{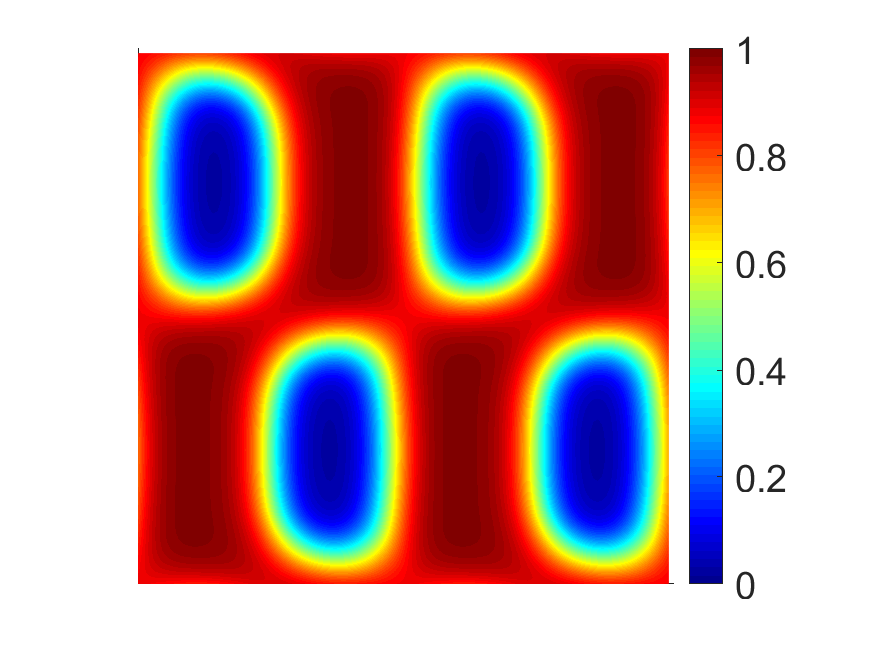}  
    &
    \includegraphics[trim={2.0cm 0.0cm 0.0cm 0.5cm},clip,scale=0.32]{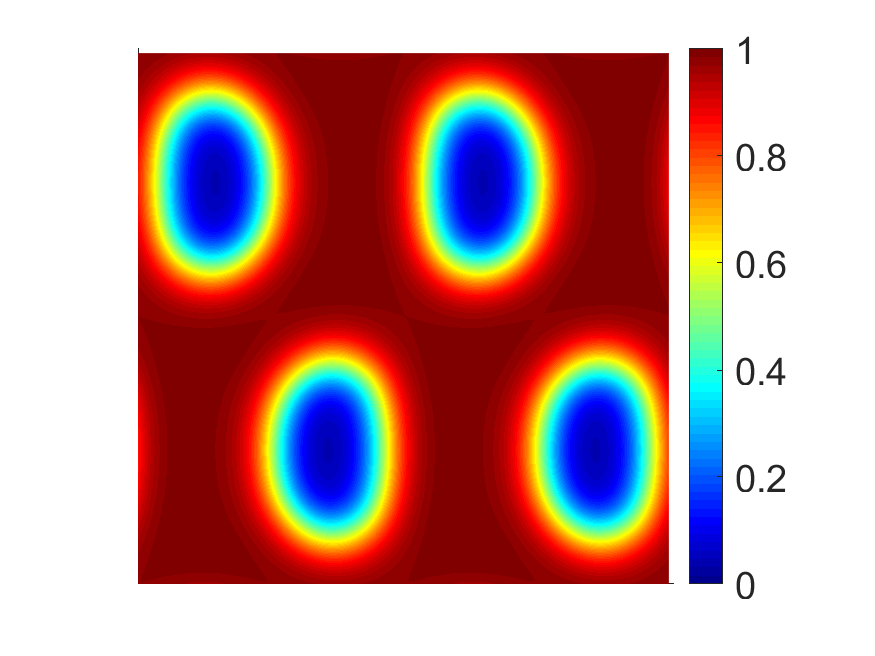}  \\[-0.5em]
     \hspace{-1.5em}t=0.0312 & \hspace{-1.5em}t=0.0625 & \hspace{-1.5em}t=0.1 \\
    \includegraphics[trim={2.0cm 0.0cm 0.0cm 0.5cm},clip,scale=0.32]{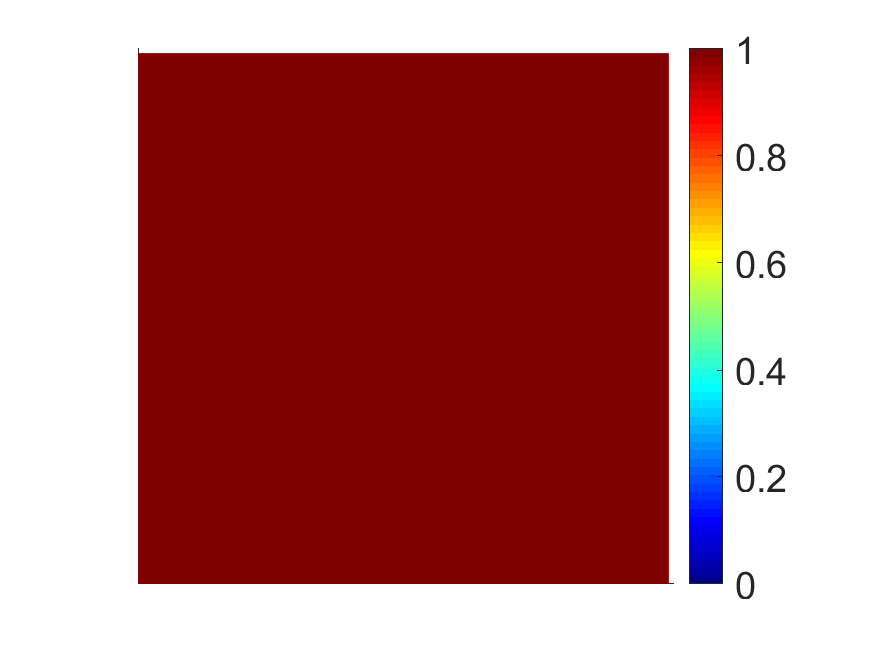}  
    &
    \includegraphics[trim={2.0cm 0.0cm 0.0cm 0.5cm},clip,scale=0.32]{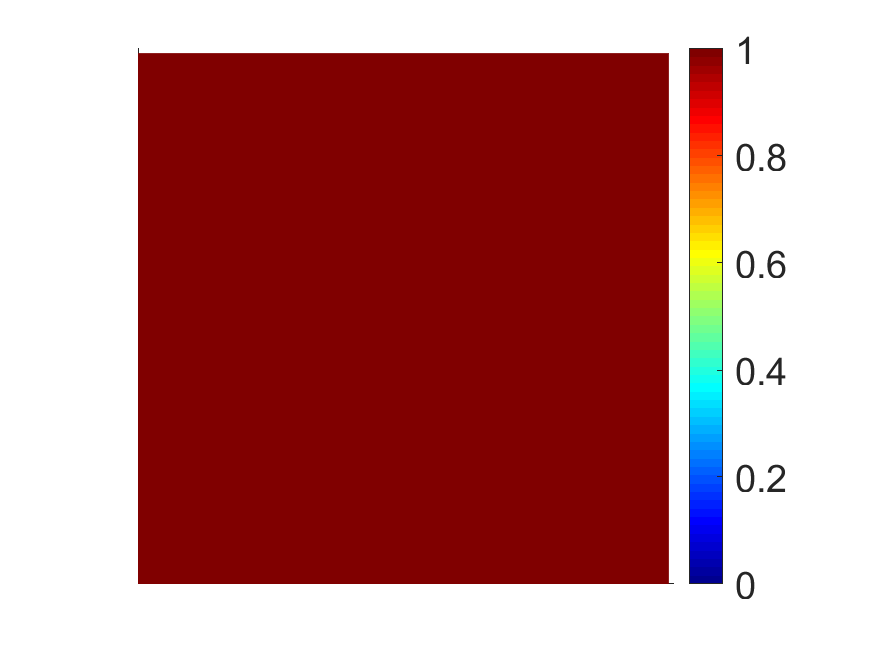} 
    &
    \includegraphics[trim={2.0cm 0.0cm 0.0cm 0.5cm},clip,scale=0.32]{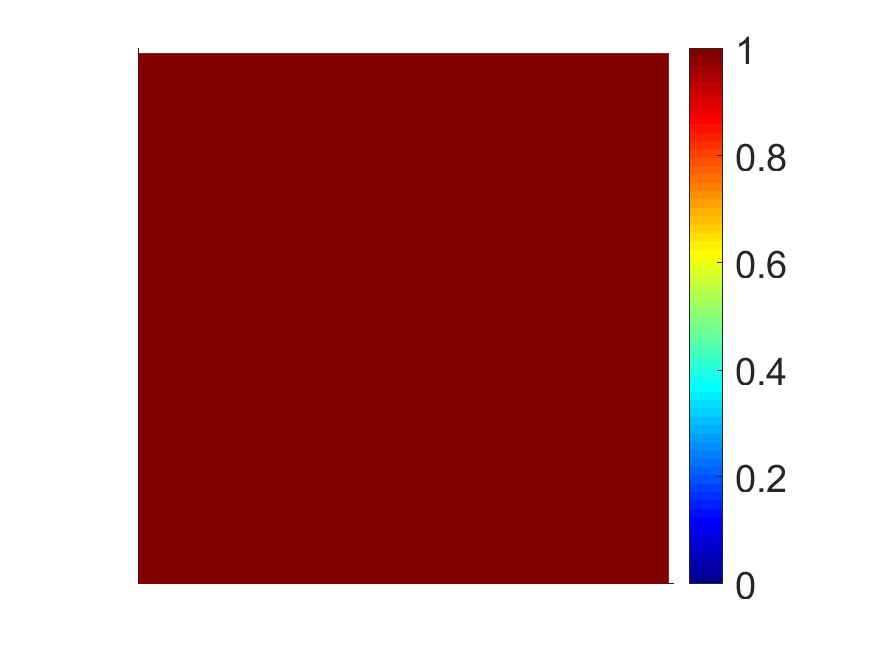} \\[-0.5em]
\end{tabular}
    \caption{Snapshots of the non-conserved phase field $\eta$ for simulations with $h=2^{-6}$ and $\tau=0.001\cdot h$.\label{fig:evoeta}}
\end{figure}
In these plots, one can observe the following typical behaviour: Both phase field variables start minimizing their surface energy creating sharper interfaces between the pure states. Due to the choice of parameters the non-conserved phase field variables evolves faster and after a short time reaches a quasi-stationary state, i.e. $\eta\approx 1$.

The deviation from the equilibrium values is then due to the cross-kinetic coupling with the conserved phase field. The evolution of the conserved phase field is much slower. In this case we can see a rather slow emergence of one connected equilibrium phase which is still far away from a stationary state.
In Figure~\ref{fig:evoerr}, we illustrate the mass-conservation and energy-dissipation of the discrete solution, which are both satisfied up to round-off errors. This was also observed for all computations in the convergence tests and we further performed tests with $c=1, 0.1, 0.01$ and the same results hold, as long as the Newton scheme converges.
\begin{figure}[htbp!]
\centering
\footnotesize
\begin{tabular}{cc}
    \includegraphics[trim={0.0cm 0.0cm 0.0cm 0.0cm},clip,scale=0.38]{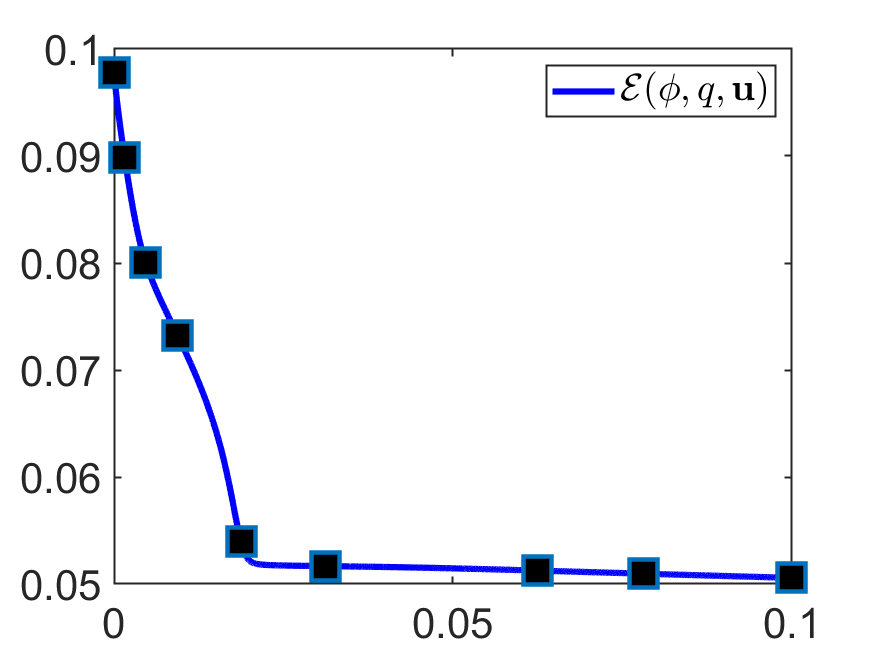} 
    &
    \includegraphics[trim={0.0cm 0.0cm 0.0cm 0.0cm},clip,scale=0.38]{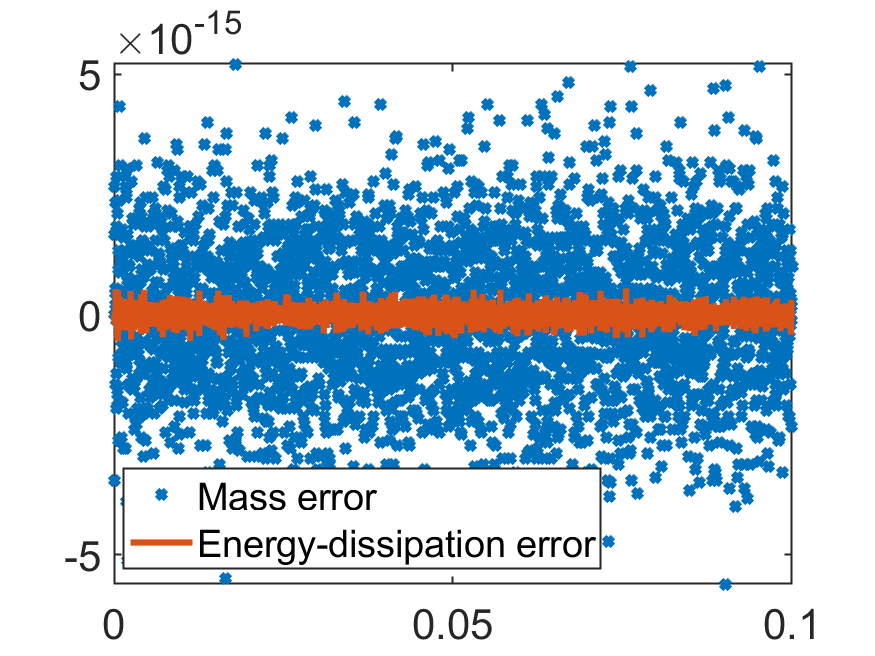}  
\end{tabular}
    \caption{Evolution of the energy $\E(\rho,\eta)$ (left) as well as validity of discrete energy-dissipation and mass conservation identities (right) for simulations obtained with $h=2^{-6}$ and $\tau=0.001\cdot h$.\label{fig:evoerr}}
\end{figure}

\subsection*{Convergence results}
We now evaluate the actual convergence rates observed in our numerical computations. 
Since no analytical solution is available, the discretization error is estimated by comparing the computed solutions $(\rho_{h,\tau},\eta_{h,\tau},\bar\mu_{\rho,h,\tau},\bar\mu_{\eta,h,\tau})$  for two consecutive uniformly refined grids.
The error quantities for the fully-discrete scheme, about which we report in the following, are defined as 
\begin{align*}
\err(h,\tau) &= \underbrace{\norm*{\rho_{h,\tau} - \rho_{h/2,\tau/2}}_{L^\infty(H^1)}}_{=\err(\rho,h,\tau)} + \underbrace{\norm*{\eta_{h,\tau} - \eta_{h/2,\tau/2}}_{L^\infty(H^1)}}_{=\err(\eta,h,\tau)} \\
&\qquad\quad+ \underbrace{\norm*{\bmu_{\rho,h,\tau} - \bmu_{\rho,h/2,\tau/2}}_{L^2(H^1)}}_{=\err(\mu_\rho,h,\tau)} + \underbrace{\norm*{\bmu_{\eta,h,\tau} - \bmu_{\eta,h/2,\tau/2}}_{L^2(L^2)}}_{=\err(\mu_\eta,h,\tau)}.
\end{align*}
According to the theoretical predictions stated in Theorem~\ref{thm:fulldisk}, each of the error components is expected to converge at second order with respect to both discretization parameters, i.e., $\err(\cdot,h,\tau)=O(h^2+\tau^2)$.
%
\begin{table}[htbp!]
\centering
\footnotesize
\caption{Errors and experimental orders of convergence. \label{tab:rates_time_chns}} 
\medskip 
%
\begin{tabular}{c||c|c||c|c||c|c||c|c}
$ k $ &  $ \err(\rho,h,\tau) $  &  eoc & $ \err(\eta,h,\tau) $  &  eoc & $ \err(\mu_\rho,h,\tau) $  &  eoc & $ \err(\mu_\eta,h,\tau) $  &  eoc \\
\hline
$ 1 $   &  $8.61\cdot 10^{+0}$ & ---   & $6.93\cdot 10^{+0}$ & ---   & $1.05\cdot 10^{-2}$ & ---   & $1.86\cdot 10^{-5}$  & ---\\
$ 2 $   & $6.77\cdot 10^{+0}$ & 0.59 & $2.48\cdot 10^{+0}$ & 1.22 & $8.35\cdot 10^{-3}$ & 0.57 & $9.71\cdot 10^{-6}$  & 0.97 \\
$ 3 $   & $1.81\cdot 10^{-1}$ & 2.29 & $1.28\cdot 10^{-1}$ & 2.07 & $4.73\cdot 10^{-4}$ & 2.04 & $4.96\cdot 10^{-7}$  & 2.07 \\
$ 4 $   & $1.22\cdot 10^{-2}$ & 1.97 & $9.66\cdot 10^{-3}$ & 1.93 & $3.16\cdot 10^{-5}$ & 1.98 & $1.13\cdot 10^{-8}$  & 2.34 \\
$ 5 $   & $8.15\cdot 10^{-4}$ & 1.97 & $6.51\cdot 10^{-4}$ & 1.97 & $2.83\cdot 10^{-6}$ & 1.87 & $7.95\cdot 10^{-10}$ & 1.96  
\end{tabular}
\end{table}
In Table~\ref{tab:rates_time_chns}, we summarize the results of our computations obtained on a sequence of uniformly refined meshes with mesh size $h=2^{-k}$, $k=1,\ldots,5$, and time steps $\tau = 0.001\cdot h$. 
The computational results are in perfect agreement with the theoretical predictions stated in Theorem~\ref{thm:fulldisk}.

\section{Discussion} \label{sec:8}
In this paper, we proposed a variational discretization approach for the Cahn-Hilliard-Allen-Cahn system with non-diagonal mobility matrix, which is based on a mixed finite-element approximation in space and a Petrov-Galerkin discretization in time. The mobility matrix is allowed to depend on both phase field variables and their spatial gradients.
Relative energy estimates were used to establish stability of the discrete problem. These estimates naturally account for the nonlinearities in the energy of the problem, which significantly simplifies the error analysis.
Let us note that a continuous version of the respective stability estimate allows to establish stability and weak-strong uniqueness also for the continuous problem; see~\cite{Brunk22}.
The chosen discretization spaces allow us to prove order optimal convergence in all variables with relaxed regularity.
In principle, the proposed schemes can be extended immediately to higher order in space and time and to more phase field variables.
Further investigations in this direction as well as the extension to more complex multiphase problems, like non-isothermal phase-field models \cite{Penrose1990}, will be topics of future research. 

}

\appendix

 \setcounter{equation}{0}
 \renewcommand{\thesection}{A}
 \renewcommand{\theequation}{\thesection.\arabic{equation}}

In the following, we summarize some well-known results about standard projection and interpolation operators and some technical facts, which are used in our analysis. 

\subsection{Space discretization}\label{asec:space}
We consider the setting of Sections~\ref{sec:prelim} and \ref{sec:main} and, in particular, assume (A5)--(A6) to hold true.
The following results then follow with standard arguments; see e.g. \cite{BrennerScott}.
The $L^2$-orthogonal projection $\pi_h^0 : L^2(\Omega) \to \Vh$,
satisfies
\begin{align} \label{eq:l2projest}
    \|u - \pi_h^0 u\|_{H^s} \leq C h^{r-s} \|u\|_{H^r} \qquad \forall u \in H^r(\Omega),
\end{align}
and all parameters $-1 \le s \le r$ and $0 \le r \le 4$. 
On quasi-uniform meshes $\Th$, which we consider here, the projection $\pi_h^0$ is also stable with respect to the $H^1$-norm, i.e., 
\begin{align} \label{eq:h1stab_l2proj}
\|\pi_h^0 u\|_{H^1} \le C \|u\|_{H^1} \qquad \forall u \in H^1(\Omega).
\end{align}
Furthermore, for every $q\in[1,\infty]$ the following stability result holds, see \cite{Douglas1974/75,crouzeix1987stability},
\begin{align}
\|\pi_h^0 u\|_{W^{1,p}} \le C \|u\|_{W^{1,p}} \qquad \forall u\in W^{1,p}(\Omega). \label{eq:l2stabw1p}   
\end{align}
The $H^1$-elliptic projection $\pi_h^1 : H^1(\Omega) \to \Vh$, defined in  \eqref{eq:defh1proj}, satisfies
\begin{align}\label{eq:h1porjest}
 \|u - \pi_h^1 u\|_{H^s} \leq C h^{r-s} \|u\|_{H^r} \qquad \forall u \in H^r(\Omega),
\end{align}
for all parameters $-1 \le s \le r$ and $1 \le r \le 3$. 
Since we assumed quasi-uniformity of the mesh $\Th$, we can further resort to the inverse inequalities
\begin{align} \label{eq:inverse}
    \|v_h\|_{H^1} \le c_{inv} h^{-1} \|v_h\|_{L^2} 
    \qquad \text{and} \qquad 
    \|v_h\|_{L^p} \le c_{inv} h^{d/p-d/q} \|v_h\|_{L^q} 
\end{align}
which hold for all discrete functions $v_h \in \Vh$ and all $1 \le q \le p \le \infty$. 

\subsection{Time discretization}\label{asec:time}
%
The piecewise linear interpolation 
$\I_\tau^1:H^1(0,T)\to P_1^c(\Itau)$ and the piecewise constant projection $\bar \pi_\tau^0 : L^2(0,T) \to P_0(\I_\tau)$ in time satisfy
\begin{align} 
    \|u - \bar\pi_\tau^0 u\|_{L^p(0,T)} &\le C \tau^{1/p-1/q+r} \|u\|_{W^{r,q}(0,T)} \quad && \forall u \in W^{r,q}(0,T), \label{eq:timprojest}\\
    \|u - \I_\tau^1 u\|_{L^p(0,T)} &\leq C\tau^{1/p-1/q+2} \|u\|_{W^{r,q}(0,T)} \quad && \forall u \in W^{s,q}(0,T) \label{eq:timinterpest}
\end{align}
with $1 \le p \le q  \le \infty$ and for $0 \le r \le 1$ respectively $1 \le s \le 2$.
Moreover, these operators commute with differentiation in the sense that
\begin{align} \label{eq:commuting} 
\dt (\I_\tau^1 u) = \bar \pi_\tau^0 (\dt u).
\end{align}

\subsection{Projection estimates for nonlinear terms}

We now derive an estimate for the projection error of products of functions in time. 

\begin{lemma} \label{lem:average_time_err}
Let $\bar a = \pi_k a$ denote the $L^2$-orthogonal projection onto $P_{0}(\Itau)$.  
Then for any $u,v \in W^{2,p}(0,T)$ with $1 \le p \le \infty$, one has 
\begin{align}
\|\overline{\bar u\bar v}-\overline{uv}\|_{L^p(0,T)} 
&\le C \tau^{2} \|u\|_{W^{2,p}(0,T)} \|v\|_{W^{2,p}(0,T)},  \label{eq:midpoint_order}
\end{align}
with a constant $C$ depending only on the polynomial degree $k$. 
\end{lemma}

\begin{lemma}
Let $\bar a = \pi_0 a$ denote the $L^2$-orthogonal projection onto $P_{0}(\Itau)$. Furthermore, let $\phi\in P_1(\Itau)$. 
Then for any $u,v \in W^{2,p}(0,T)$ with $1 \le p \le \infty$, one has 
\begin{align}
\| g(\bar\phi)-\overline{g(\phi)}\|_{L^p(0,T)} 
&\le C \tau^{2} \|g(\phi)\|_{W^{2,p}(0,T)},  \label{eq:midpoint_order_single}
\end{align}
with a constant $C$ depending only on certain continuous embedding constants. 
\end{lemma}

\subsection{Discrete Gronwall lemma}

In our stability analysis, we also employ the following well-known argument, whose proof follows immediately by induction.
\begin{lemma} \label{lem:discgronwall}
Let $(a^n)_n$, $(b_n)_n$, $(c_n)_n$, and $(\lambda_n)_n$ be given positive sequences, satisfying
\begin{align*}
    u_n + b_n \le e^{\lambda_n} u_{n-1} + c_n, \qquad n \ge 0.
\end{align*}
Then 
\begin{align} \label{eq:discgronwall}
    u_n + \sum_{k=1}^n e^{\sum_{j={k+1}}^{n} \lambda_j} b_k 
    \le e^{\sum_{j=1}^n \lambda_j} u_0 + \sum_{k=0}^n e^{\sum_{j={k+1}}^{n} \lambda_j} c_k, \quad n > 0. 
\end{align}
\end{lemma}

{
\bigskip 

\section*{Acknowledgement}
Support by the German Science Foundation (DFG) via SPP~2256: \textit{Variational Methods for Predicting Complex Phenomena in Engineering Structures and Materials} (projects Eg-331/2-1 and BR~7093/1-2) and via TRR~146: \textit{Multiscale Simulation Methods for Soft Matter Systems} (project C3) is gratefully acknowledged. Part of the research was conducted during a research stay of the first author at RICAM/JKU Linz.}


\end{document}